\documentclass[a4paper,11pt,x11names]{article}
\usepackage[T1]{fontenc}
\usepackage[utf8x]{inputenc} 
\usepackage{lmodern}
\usepackage{amsmath}
\usepackage{amsthm}
\usepackage{amssymb}
\usepackage{authblk}
\usepackage{graphicx}
\usepackage{setspace}
\usepackage{enumitem}

\usepackage{amsthm}
\usepackage{amssymb}
\usepackage{xparse}
\usepackage{thmtools}
\usepackage{stackrel}
\usepackage{bbold}

\newcommand{\R}{\mathbb{R}}
\newcommand{\C}{\mathbb{C}}

\newcommand{\N}{\mathbb{N}}

\newcommand{\dd}{\mathrm{d}}

\renewcommand{\S}{\mathbb{S}}

\DeclareMathOperator{\Hess}{Hess}

\DeclareMathOperator{\supp}{supp}

\DeclareMathOperator{\Sp}{Sp}
\DeclareMathOperator{\tr}{tr}

\DeclareMathOperator{\Ker}{Ker}

\DeclareMathOperator{\dist}{dist}
\DeclareMathOperator{\Span}{Span}
\DeclareMathOperator{\diam}{diam}

\makeatletter
\newtheoremstyle{indented}
{7pt} 
{7pt} 
{} 
{1.5em} 
{\bfseries} 
{.} 
{.5em} 
{} 

\theoremstyle{definition}

\newtheorem{defn}{Definition}[section]
\newtheorem{ex}[defn]{Example}
\newtheorem{example}[defn]{Example}
\theoremstyle{plain}
\newtheorem*{theorem*}{Theorem}

\newtheorem{theorem}{Theorem}

\newenvironment{preuve}{
	\noindent \textbf{Proof. }}{\hfill $\square$\medskip\par}

\newtheorem{prop}[defn]{Proposition}

\newtheorem{lem}[defn]{Lemma}

\theoremstyle{definition}
\newtheorem{rem}[defn]{Remark} 

\makeatletter
\renewcommand*\env@matrix[1][*\c@MaxMatrixCols c]{%
  \hskip -\arraycolsep
  \let\@ifnextchar\new@ifnextchar
  \array{#1}}
\makeatother

\usepackage{a4wide}
\title{Low-energy spectrum of Toeplitz operators with a miniwell}
\author{Alix Deleporte\thanks{deleporte@math.unistra.fr}}
\affil{Universit\'e de Strasbourg, CNRS, IRMA UMR 7501, F-67000 Strasbourg, France}

\begin{document}

\maketitle

\let\thefootnote\relax\footnote{This work was supported by grant
  ANR-13-BS01-0007-01}
\begin{abstract}
We study the concentration properties of low-energy states for quantum
systems in the semiclassical limit, in the setting of Toeplitz
operators, which include quantum spin systems as a large class of examples. We establish tools proper to Toeplitz quantization to give
concentration properties under weak conditions. In addition, we build
up symplectic normal forms in two particular settings, including a generalisation of
Helffer-Sj\"ostrand miniwells, in order to prove asymptotics for the
ground state and estimates on the number of low-lying eigenvalues.

  
\end{abstract}
\section{Introduction}
 \subsection{Quantum selection}
The computation of ground states for quantum systems is an ubiquitous
problem of great difficulty in the non-integrable case, such as
antiferromagnetic spin models on lattices in several dimensions.
On those systems, approaches in the large spin limit are commonly used
\cite{sobral_order_1997,lecheminant_order_1997,reimers_order_1993,chubukov_order_1992}, in
an effort to reduce the problem to the study of the minimal set of the
classical energy. A general procedure of \emph{semiclassical order by
  disorder} was proposed by Dou\c{c}ot and
Simon \cite{doucot_semiclassical_1998}, in situations where this
classical minimal set is not discrete.

In the mathematical setting of Schr\"odinger operators in the
semiclassical limit, a general study of ground state properties 
was done by Helffer and
Sj\"ostrand \cite{helffer_multiple_1984,helffer_puits_1986}, including situations where the minimal set of the
potential is a smooth submanifold.
The classical phase space of spin systems, a product of spheres, is
compact. In particular, spin systems are neither Schr\"odinger operators
nor given by Weyl quantization. However, spin operators are example of
Toeplitz operators, which allows to understand the large spin limit as
a semiclassical limit. In a previous article
\cite{deleporte_low-energy_2016}, we studied semiclassical
concentration of ground states in the context of semiclassical Toeplitz operators, when the
minimal set of the classical energy (or symbol) is a finite set of non-degenerate
points, with results analogous to the Schr\"odinger case
\cite{helffer_multiple_1984}.

In frustrated antiferromagnetic spin systems, such as on the Kagome
lattice, the minimal set of the classical energy does not form a smooth
submanifold. The goal of this article is to not only to extend the
degenerate case \cite{helffer_puits_1986} to Toeplitz quantization, but also to generalise the
geometrical conditions on the zero set of the classical energy.

We prove several results of \emph{quantum selection}: not all points
of classical phase space where the energy is minimal are equivalent
for quantum systems; and the semiclassical quantum ground state
localises only on a subset of the classical minimal set. To do so, on
one hand we
develop techniques which are proper to Toeplitz quantization; on the
other hand we prove new symplectic normal forms which are also useful in
the context of pseudodifferential calculus.

Theorem A requires
a weak condition on the symbol near its zero set, and applies in
particular to symbols which are analytic near zero, which is the case
of all physical examples of interest. It is done in the spirit of Melin's inequality (see
\cite{melin_lower_1971}, or \cite{hormander_analysis_2007}, Thm
22.2.3). 
Theorem B is more precise and applies in a particular
setting which generalises the miniwells of
\cite{helffer_puits_1986}. Theorem C treats a degenerate case where
the symbol is minimal on a set with a singular point. Theorem D
analyses the relative role of regular and singular points in the low-energy Weyl law.

\subsection{Main results}

In order to state the main theorems we need to introduce Toeplitz
quantization, the criterion
under which localisation takes place, and what localisation means in
this context.

Toeplitz quantization takes place on quantizable K\"ahler manifolds
\cite{charles_berezin-toeplitz_2003}, which are complex manifolds with
a symplectic and a Riemannian structure. Let $M$ be such a manifold and let $h\in C^{\infty}(M,\R)$. Through Toeplitz quantization we associate to this
function a sequence of self-adjoint operators $(T_N(h))_{N\geq 1}$
acting on a sequence of Hilbert spaces $(H_N)_{N\geq 1}$ (see
Definitions \ref{defn:szego} and \ref{def:toeplitz}). The semiclassical limit
is $N\to +\infty$. In this article we are interested in
the spectrum and eigenvectors of $T_N(h)$ for $N$ large.

Suppose $\min(h)=0$. The selection criterion
consists in a continuous function $\mu$ (see Definition
\ref{def:mu}). This function is defined on $\{h=0\}$, and depends on
the Hessian of $h$. It captures the effects of order $N^{-1}$ on the low-energy spectrum of $T_N(h)$. For each point $x$ such
that $h(x)=0$, we call $\mu(x)$ the \emph{Melin value} at $x$.

 The Hilbert spaces $H_N$ consist of $L^2$ functions on a bundle
over $M$ (with projection $\pi$). In particular, for any Borel set $B\subset M$ and function $u\in
H_N$, the microlocal mass of $u$ on $B$ is directly
defined (as $\int_{\pi^{-1}(B)}|u|^2\dd Vol$), in contrast with Weyl quantization where this needs some
work; see Definition \ref{defn:loc}.
\begin{theorem}
  Let $M$ be a compact K\"ahler quantizable manifold and $h\in
  C^{\infty}(M,\R)$. Suppose that $\min(h)=0$. Let $\mu$ the function
  associating to each point where $h$ vanishes the Melin value at this
  point. Let $\mu_{\min}=\min(\mu(x),x\in M,h(x)=0)$.  
  Suppose that there exist $C>0$ and $\alpha>0$ such that, for every $t\geq 0$,
one has
\[
\dist_{\text{Hausdorff}}\left(\{h\leq t\},\{h=0\}\right)\leq C t^{\alpha}.
\]

Then there exist $C>0$ and $\epsilon>0$ such that, for every $N\geq
1$, one has
\[|\min \Sp(T_N(h)) - N^{-1}\mu_{\min}| \leq CN^{-1-\epsilon}.\]
Here $\Sp(T_N(h))$ denotes the spectrum of $T_N(h)$.

Let $((\lambda_N,u_N))_{N\geq 1}$ be a sequence of
eigenpairs of $(T_N(h))_{N\geq 1}$. If $\|u_N\|_{H_N}=1$ and
$\lambda_N=N^{-1}\mu_{min}+o(N^{-1})$, then
for any open set $U$ at positive distance from $$\{x\in M, h(x)=0,
\mu(x)=\mu_{\min}\},$$ as $N\to +\infty$ there holds $$\int_{\pi^{-1}(U)}
|u_N|^2\dd Vol=O(N^{-\infty}).$$
\end{theorem}
Theorem A already appears in previous work
\cite{deleporte_low-energy_2016}, under the much stronger hypothesis
that $\{h=0\}$ is a finite set of regular critical points.

\begin{theorem}
  Under the hypotheses of Theorem A, suppose that the function $\mu$ reaches its non-degenerate minimum on a unique
  point $P_0$. Suppose further that, in a neighbourhood of $P_0$, the
  set $\{h=0\}$ is an isotropic submanifold of $M$, on which $h$
  has non-degenerate transverse Hessian matrix.

Then for any sequence $(u_N)_{N\geq 1}$ of unit eigenfunctions corresponding to the first
eigenvalue of $T_N(h)$, for any $\epsilon>0$, one has
$$\int_{\big\{dist(\pi(y),P_0)>N^{-\frac 14+\epsilon}\big\}}|u_N(y)|^2\dd Vol=O(N^{-\infty}).$$

Moreover, the first eigenvalue is simple and the spectral gap is of
order $N^{-\frac 32}$. There is a full expansion of the first
eigenvalue and eigenvector in powers of $N^{-\frac 14}.$
\end{theorem}
Following Helffer-Sj\"ostrand \cite{helffer_puits_1986}, we will call
$P_0$ a \emph{miniwell} for $h$.

Under the conditions of Theorem B, the first eigenvector concentrates
rapidly on $\{h=0\}$, and the speed of concentration towards the point
which minimises $\mu$ is much slower. In particular this state is more
and more squeezed as $N$ increases.

\begin{theorem}
  Under the hypotheses of Theorem A, suppose that the function $\mu$  reaches its minimum on a
  unique point $P_0$ at which there is a \emph{simple crossing} (see
  Definition \ref{defn:simple-crossing}).

  Then for any sequence $(u_N)_{N\geq 1}$ of unit eigenfunctions corresponding to the first
eigenvalue of $T_N(h)$, for any $\epsilon>0$, one has
$$\int_{\big\{dist(\pi(y),P_0)>N^{-\frac 13+\epsilon}\big\}}|u_N(y)|^2\dd Vol=O(N^{-\infty}).$$

Moreover, the first eigenvalue is simple and the spectral gap is of
order $N^{-\frac 43}$. There is a full expansion of the first
eigenvalue and eigenvector in powers of $N^{-\frac 16}.$
\end{theorem}

An example of symbol with a \emph{simple crossing}, with dimensions $(1,1)$, is the following
function on $\R^{4}$:
\[
h:(q_1,q_2,p_1,p_2)\mapsto p_1^2+p_2^2+q_1^2q_2^2,\]
which reaches its minimum on the transverse union of two manifolds,
$\R\times \{0,0,0\}$ and $\{0\}\times \R\times \{0,0\}$, intersecting
at one point.

As in the case of Theorem $B$, the first eigenvector is more and more
squeezed as $N\to +\infty$. Note that the speed of convergence, and
the powers of $N$ involved in the expansions, differ between the two cases.

The question now arises of the inverse spectral
problem in our setting: given the high $N$ spectrum of a Toeplitz
operator, is one able to distinguish the geometry of the set on which
the Melin value $\mu$ is minimal?

\begin{theorem}
  Let $h\in C^{\infty}(M,\R)$ with $\min(h)=0$. There exist $0<c \leq
  C$, $\epsilon>0$
  and $N_0\geq 0$
  such that the following is true. Let $\mu_{\min}$
  be the infimum of the Melin value, and
  $N\geq N_0$.
\begin{enumerate}[label=\Alph*.]
  \item Any eigenfunction of $T_N(h)$ associated with an eigenvalue in the spectral
  window $[0,\mu_{\min}N^{-1}+\epsilon N^{-1}]$ is localised on
  a small neighbourhood of the set of minimal Melin value.

  \item For each regular miniwell with Melin value $\mu_{\min}$ and
  dimension $r$, for each sequence $(\Lambda_N)$ with
\[
  N^{-\frac 12+\epsilon}\leq \Lambda_N \leq \epsilon,
\]
in the
  spectral window $[0,N^{-1}(\mu_{\min}+\Lambda_N)]$, the
  number of orthogonal almost eigenfunctions of $T_N(h)$ supported on
  a small neighbourhood of the miniwell
  belongs to the interval $$\left[c(N^{\frac 12} \Lambda_N)^r,C(N^{\frac 12}\Lambda_N)^r\right].$$

  \item For each simple crossing with Melin value $\mu_{\min}$ and
  dimensions $(r,r)$,  for each sequence $(\Lambda_N)$ with
\[
N^{-\frac 13+\epsilon}\leq \Lambda_N \leq \epsilon,
\]
in the
  spectral window $[0,N^{-1}(\mu_{\min}+\Lambda_N)]$, the
  number of orthogonal almost eigenfunctions of $T_N(h)$ supported on
  a small neighbourhood of
  the crossing point belongs to the interval $$\left[c(N^{\frac 13}\Lambda_N)^{\frac{3r}{2}}\log(N^{\frac 13}\Lambda_N),C(N^{\frac 13}\Lambda_N)^{\frac{3r}{2}}\log(N^{\frac 13}\Lambda_N)\right].$$
  \end{enumerate}
\end{theorem}
The notion of dimension of a miniwell and a simple crossing can be
found in Definition \ref{defn:dim}.
In Theorem D, case A is a generalisation of
Theorem A. Cases B and C apply respectively in the settings of
Theorems B and C.

\begin{rem}
  If $\Lambda_N<N^{-\epsilon}$, then there are more eigenvalues near a
  miniwell than near a crossing point (the ratio is of order
  $N^{\frac{\epsilon}{2}})$. If we look at eigenvalues in such
  windows, then a miniwell of dimension $r$ not only ``hides''
  miniwells of smaller dimensions, but also crossing points of
  dimensions up to and including $(r,r)$.

  If $\Lambda_N>\frac{\epsilon}{2}$, then there are more
  eigenvalues near a crossing point than near a miniwell (the ratio is
  of order $\log(N)$). In these windows, crossing points hide
  miniwells of dimension smaller or equal.

  In particular, this proves that the spectral inverse problem allows,
  not only to recover the value of $\mu_{\min}$, but also to determine
  the largest dimensions of the miniwells or crossing points achieving
  $\mu_{\min}$, and to tell whether there are only miniwells, only
  crossing points, or both.
\end{rem}
Theorem D also allows to study low-temperature quantum
states for a model on which there is a competition between a regular
point and a crossing point with the same $\mu$. It shows a transition
from temperature ranges similar to $N^{-1}$, for which the Gibbs
measure concentrates on the crossing point, and temperature ranges of
order $N^{-1-\epsilon}$, for which this measure concentrates on the
regular point.

In this work only rapid decay estimates are
obtained: quantities are controlled modulo an $O(N^{-\infty})$ error
as $N\to +\infty$. The natural question of exponential decay
\cite{helffer_puits_1986} requires
refined estimates on the Szeg\H{o} kernel which are currently
unknown for general compact K\"ahler manifolds.

The study of the function $\mu$ on examples of spin systems requires the full
diagonalisation of matrices which size grows with the number of
sites. Few theoretical results are known in this setting (see Section
8). The general conjecture is that $\mu$ should reach a minimum on
planar configurations; up to now this is only supported by numerical
evidence and the fact that planar configurations are local minima for $\mu$.

\subsection{Application to spin systems}
\label{sec:appl-spin-syst}

One of the main physical motivations for this study, discussed in
detail in Section 8, is the
mathematical foundation of \emph{quantum selection} in the context of
spin systems. The search for materials with a non-conventional
magnetic behaviour led experimental and theoretical physicists to consider frustrated
antiferromagnetic spin systems, such as pyrochlore or the Kagome
lattice. \emph{Order by disorder} approaches in the large spin limit
are commonly used in the physics literature, and the subprincipal effects presumably select a
very small subset of configurations
\cite{doucot_semiclassical_1998,sobral_order_1997,lecheminant_order_1997,reimers_order_1993,chubukov_order_1992}.

Spin systems are particular cases of Toeplitz operators. In such systems the base manifold is a product of
$2$-spheres. Let $G=(V,E)$ be a finite graph and $M=(\S^2)^{\times
  |V|}$. At each vertex of the graph one associates a unit vector in
$\R^3$. Let us consider the $3|V|$ real functions associating to a
given vertex $i\in V$, the coordinates $x_i,y_i,z_i$ of the associated
unit vector $e_i$. The symplectic structure on $M$ is such that
$\{x_i,y_j\}=\delta_{ij}z_i$; two similar identities hold by cyclic
permutation. We introduce the \emph{antiferromagnetic Heisenberg}
symbol:

$$\begin{matrix}
  h:&M&\mapsto& \R\\
  &(e_i)_{i\in V}&\mapsto& \sum\limits_{(i,j)\in E}x_ix_j+y_iy_j+z_iz_j.\end{matrix}$$

The classical minimum of this function corresponds to situations where
the sum of the scalar products between neighbouring vectors is the
smallest. If $G$ is bipartite, this minimum is reached in situations
where neighbouring vectors are opposite. In \emph{frustrated systems},
this is not possible. If for instance three vertices in the graph are
linked with each other, then not all of them can be opposite to the
other ones. This is the case of the Kagome lattice, and the
Husimi tree, considered in \cite{doucot_semiclassical_1998} and
depicted in Figure \ref{fig:Huskag}.

We will consider a class of graphs \emph{made of triangles}. A finite
connected graph
$G=(E,V)$ is \emph{made of triangles} when there is a partition $V=\bigsqcup_{i\in J}V_i$
where, for every $i$, $V_i$ contains three edges that link together
three vertices; in addition, we ask that the degree at any vertex does not
exceed 4 (and is hence equal to either $2$ or $4$). We will call the $V_i$'s
the triangles of the graph.

Finite subgraphs of the Kagome lattice and the Husimi tree of Figure \ref{fig:Huskag} are
made of triangles. In general, from a $3$-regular finite graph $G=(V,E)$, one
can build an associated graph made of triangles
$\widetilde{G}=(\widetilde{V},\widetilde{E})$ which is the \emph{edge graph} of $G$: the set of
vertices is $\widetilde{V}=E$ and two elements
of $\widetilde{V}$ are adjacent in $\widetilde{G}$ when they are adjacent as
edges of $G$ (i.e. when they share a common vertex). In this case the
triangles of $\widetilde{G}$ correspond to the vertices of $G$. The Kagome
lattice is thus associated with the hexagonal lattice, and the
infinite Husimi
tree with the $3$-regular tree.

\begin{figure}
   \centering
   \includegraphics[scale=0.1]{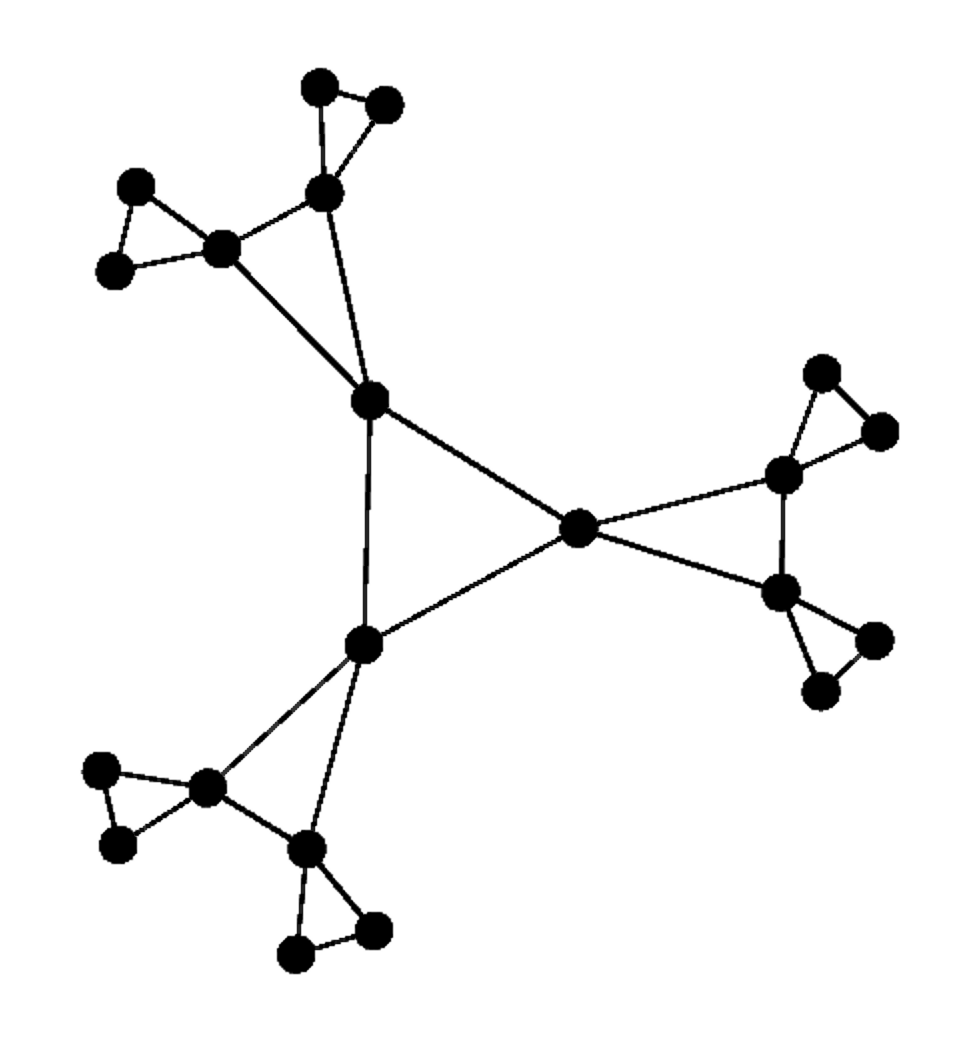}
   \hspace{2em}
   \includegraphics[scale=0.06]{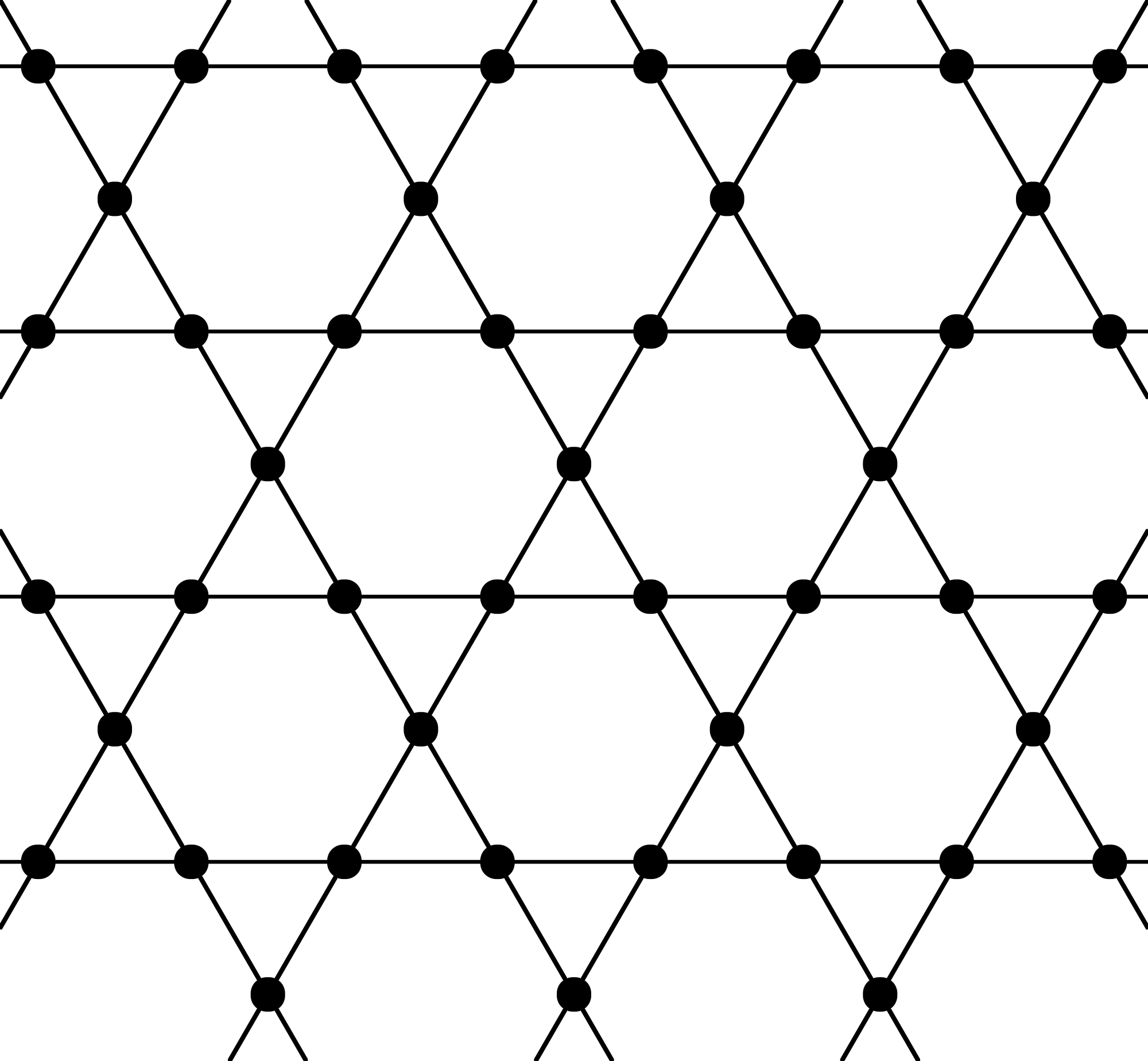}
   \hspace{2em}
   \includegraphics[scale=0.3]{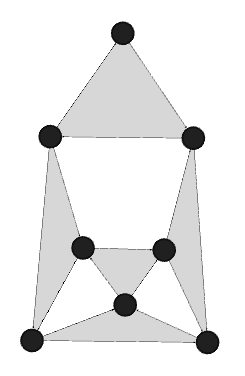}   
   \caption{Main examples: a piece of the Husimi tree (left), and the Kagome
     lattice (middle). On the right, a graph made of $5$ triangles on
     which the symbol cannot reach $-15/2$.}
   \label{fig:Huskag}
 \end{figure}

 The presence of the ``frustration'' by triangles leads to a large
 degeneracy of the classical minimal set. Indeed, $h$ is minimal if,
 on each triangle $V_i$, the sum of the three spins at the vertices of
 the triangle is zero (so that these elements of $\S^2$ must form a
 great equilateral triangle). This is not always possible as the
 example to the right of Figure \ref{fig:Huskag} shows. Those
 configurations exist on subsets of the Husimi tree and the Kagome
 lattice, and are highly degenerate: on the Husimi tree, once the
 spins on a triangle are chosen, there is an $\S^1$ degeneracy for
 each of its children; the set of minimal configurations is an
 isotropic torus whose dimension grows linearly with the number of
 triangles. On the Kagome lattice, the set of these configurations
 does not form a smooth submanifold, hence the need for Theorems A and
 C. It is currently unknown which minimal points of $h$ achieve
 $\mu_{\min}$.

 \begin{prop}
   For a loop of 6 triangles (the basic element of the Kagome
   lattice), the minimal set is not a smooth manifold.

   For a loop of 4 triangles, the minimal set is the direct product of
   $SO(3)$ and the union of three circles, two of each having
   transverse intersection at exactly one point. Planar configurations
   are local minima for $\mu$.
 \end{prop}
 The proof is presented in Section 8.

\subsection{Outline}
In Section 2 we recall the necessary material on Toeplitz operators,
including a universality lemma proved in \cite{deleporte_low-energy_2016}, and
quantum maps as developed in \cite{charles_semi-classical_2007}. In
this section we also define and study the Melin value.

Section 3 contains the main tool in the proof of Theorem A, which is a Toeplitz
version of the Melin estimates \cite{melin_lower_1971,hormander_analysis_2007}. We give a global and a local version
of these estimates, and use them to prove pseudolocality of the
resolvent at a distance $\geq \varepsilon N^{-1}$ of the spectrum, for
every $\varepsilon>0$.

Section 4 concludes the proof of Theorem A, based on the Melin estimate.

Sections 5 and 6 respectively contain the proofs of Theorems B and C,
following the same strategy. We first find a convenient
symplectic normal form, then use quantum maps to reduce
the problem to a normal form, then we use a particular perturbative
argument to exhibit, for every $k\in \N$, an approximate eigenvector up to $O(N^{-k})$.

In Section 7 we prove Theorem D, using estimates developed in the
previous sections, and especially Proposition \ref{prop:Weylreg} and
Proposition \ref{prop:Weylcross}.

Section 8 contains a discussion on frustrated spin systems.


\section{Toeplitz operators}
Toeplitz operators are a generalisation of the Bargmann-Fock point of
view on the quantum harmonic oscillator
\cite{bargmann_hilbert_1961}. They realise a \emph{quantization} on
some symplectic manifolds, and are a particular case of geometric
quantization
\cite{kostant_quantization_1970,souriau_quantification_1967}. Another
particular case of geometric quantization is Weyl quantization which
leads to pseudodifferential operators.
Toeplitz operators were first studied from a microlocal point of view
\cite{boutet_de_monvel_spectral_1981,boutet_de_monvel_sur_1975}, and
the study of the Szeg\H{o} kernel was further motivated by
geometrical applications
\cite{demailly_holomorphic_1991,zelditch_szego_2000}. Here we directly
use the semiclassical point of view developed in
\cite{shiffman_asymptotics_2002,charles_berezin-toeplitz_2003,ma_holomorphic_2007}.

In this Section we recall the properties of Toeplitz operators, and we
refer to earlier work on the topic
\cite{deleporte_low-energy_2016,berman_direct_2008,ma_holomorphic_2007,charles_berezin-toeplitz_2003,shiffman_asymptotics_2002,woodhouse_geometric_1997,bargmann_hilbert_1961}
for the proofs of the exposed facts.
\subsection{Hardy spaces and the Szeg\H{o} projector}
If a symplectic manifold (phase space) $M$ has a complex structure,
the idea behind Berezin-Toeplitz operators is to consider quantum
states as holomorphic functions. If $M$ is compact,
holomorphic functions on $M$ are all constant, so that one needs to
consider sections of a convenient line bundle over $M$ or,
by duality, holomorphic functions on a dual line bundle.

Let $M$ be a K\"ahler manifold of dimension $n$, with symplectic form
$\omega$. If the Chern class of $\omega/2\pi$ is integral, there
exists a Hermitian holomorphic line bundle $(L,h)$ over $M$, with
curvature $-i\omega$ (\cite{woodhouse_geometric_1997}, pp. 158-162).

Let $(L^*,h^*)$ be the dual line bundle of $L$, with dual metric. Let $D$ be the unit ball of $L^*$, that is:
$$\{D=(m,v)\in L^*, \|v\|_{h^*}<1\}.$$ The boundary of $D$ is denoted
by $X$. It is a circle bundle over $M$, with projection $\pi$ and an $\S^1$ action
\begin{align*}
r_{\theta}: \quad X\quad &\mapsto \quad X\\(m,v)&\mapsto (m,e^{i\theta}v).
\end{align*}

$X$ inherits a Riemannian structure from $L^*$ so that $L^2(X)$ is
well-defined. We are interested in the equivariant Hardy spaces on $X$, defined as follows:
\begin{defn}\label{defn:szego}$\,$
	\begin{itemize}
		\item The \emph{Hardy space} $H(X)$ is the closure in $L^2(X)$ of $$\{f|_X,\,f\in C^{\infty}(D\cup X),\, f\text{ holomorphic in $D$}\}.$$
	
	\item The Szeg\H{o} projector $S$ is the orthogonal projection from $L^2(X)$ onto $H(X)$.
	
	\item Let $N\in \mathbb{N}$. The \emph{equivariant Hardy space} $H_N(X)$ is:
	$$H_N(X)=\{f\in H(X),\,\forall (x,\theta)\in X\times \S^1,\,f(r_{\theta}x)=e^{iN\theta}f(x)\}.$$
	
	\item The equivariant Szeg\H{o} projector $S_N$ is the orthogonal projection from $L^2(X)$ onto $H_N(X)$.
	\end{itemize}
      \end{defn}

      Throughout this paper, we will work with the sequence of spaces
$(H_N(X))_{N\in \N}$. If $M$ is compact, then the spaces $H_N(X)$ are
finite-dimensional spaces of smooth functions. (Note, however, that
this dimension grows polynomially with $N$.) Hence, the Szeg\H{o}
projector has a Schwartz kernel, that we will also denote by $S_N$.

\begin{ex}[The sphere]\label{ex:CP1}
  The sphere $\S^2$ has a canonical K\"ahler structure as
  $(\C\mathbb{P}^1,\omega_{FS})$, which is quantizable. Here $D$ is
  the unit ball in $\C^2$, blown up at zero, and $X=\S^3$. One
  recovers the usual $\S^1$ free action on $\S^3$ with quotient
  $\S^2$.

  Here, $H_N(X)$ is the space of homogeneous polynomials of two
  complex variables, of degree $N$, with Hilbert structure the scalar
  product of the restriction to $X$ of these polynomials. A natural
  Hilbert basis corresponds to the normalized monomials
  \[(z_1,z_2)\mapsto \sqrt{\frac{N+1}{\pi}\binom{N}{k}}z_1^kz_2^{N-k}.\]
  In particular, the Szeg\H{o} projector has kernel
  \[
    S^{\C\mathbb{P}_1}_N:(z,w)\mapsto \cfrac{N+1}{\pi}(z\cdot
    \overline{w})^N.
    \]
\end{ex}

\begin{ex}[$\C^n$]
  Another important example (though non compact) is the case $M=\C^n$,
with standard K\"ahler form. As $\C^n$ is contractile, the bundle $L$
is trivial, but the metric $h$ is not. The curvature condition yields:
\[
(L,h)=\left(\C^{n}_z\times \C_v, e^{|z|^2}|v|^2\right).
\]
This leads to the following identification \cite{bargmann_hilbert_1961}:
	$$H_N(X)\simeq B_N:=L^2(\C^n)\cap \left\{z\mapsto e^{-\frac N2|z|^2}f(z),\, f\text{ is an entire function}\right\}.$$
	
	The space $B_N$ is a closed subspace of $L^2(\C^n)$. The orthogonal projector $\Pi_N$ from $L^2(\C^n)$ to $B_N$ admits as Schwartz kernel the function
	$$\Pi_N:(z,w)\mapsto
        \left(\cfrac{N}{\pi}\right)^n\exp\left(-\frac 12
          N|z-w|^2+iN\Im(z\cdot \overline{w})\right).$$
        As the case $M=\C^n$ is of particular interest, we will keep
       separate notations for the Szeg\H{o} kernel in this case, which will
        always be denoted by $\Pi_N$.
\end{ex}

The sequence of kernels $(\Pi_N)_{N\geq 1}$ is rapidly decreasing outside the
diagonal set. This key property also holds in the case of a compact K\"ahler manifold:

\begin{prop}[\cite{ma_holomorphic_2007}, prop 4.1, or \cite{charles_berezin-toeplitz_2003,shiffman_asymptotics_2002,berman_direct_2008}]\label{prop:ma4.1}
	Let $M$ be a compact K\"ahler manifold, and $(S_N)_{N\geq 1}$
        be the sequence of Szeg\H{o} projectors of Definition
        \ref{defn:szego}. Let $\delta \in [0,1/2)$. For every $k\geq
        0$ there exists $C$ such that, for every $N\geq 1$, for every $x,y\in X$ such that $\dist(\pi(x),\pi(y))\geq N^{-\delta}$, one has $$|S_N(x,y)|\leq CN^{-k}.$$
\end{prop}

This roughly means that, though the operators $S_N$ are non-local, the ``interaction range'' shrinks with $N$.

In the spirit of Proposition~\ref{prop:ma4.1}, we define what it means
for a sequence of functions in $H_N(X)$ to be localised on a set.
\begin{defn}\label{defn:loc}
	Let $u=(u_N)_{N\geq 1}$ be a sequence of unit elements of $L^2(X)$. Let $\dd Vol$ denote the Liouville volume form on $M$. For every $N$, the probability measure $|u_N|^2\dd Vol \otimes \dd \theta$ is well-defined on $X$, and we call $\mu_N$ the push-forward of this measure on $M$.
	
	Let moreover $Z\subset M$ be compact. We will say that the sequence $u$
        localises on $Z$ when, for every open set $U \subset M$
        at positive distance from $Z$, one has, as $N\to +\infty$:
	$$\mu_N(U)=O(N^{-\infty}).$$
\end{defn}

A corollary of this definition is that, if a sequence $(u_N)_{N\geq
  1}$ localises on a set $Z$, then so does the sequence
$(S_Nu_N)_{N\geq 1}$.

\begin{rem}
  Elements in the Hardy space are functions on the whole phase
  space. Hence, Definition \ref{defn:loc} corresponds to
  \emph{micro}localisation for elements of $L^2(\R^n)$.
\end{rem}

To complete Proposition \ref{prop:ma4.1}, we have to describe how
$S_N$ acts on sequences of functions concentrated on a point. For this we need a convenient choice of coordinates.

Let $P_0\in M$. The real tangent space $T_{P_0}M$ carries a Euclidian
structure and an almost complex structure coming from the K\"ahler
structure on $M$. Then, we can (non-uniquely) identify $T_{P_0}M$ with
$\C^n$ endowed with the standard metric.

The map $\exp_M:T_{P_0}M\mapsto M$ on the Riemannian manifold $M$,
together with the identification $\C^n\simeq T_{P_0}M$, leads to the
notion of normal coordinates:
\begin{defn}\label{defn:normap}
	Let $U$ be a neighbourhood of $0$ in $\C^n$ and $V$
	be a neighbourhood of a point $P_0$ in $M$.
	
	A smooth diffeomorphism $\rho:U\times \S^1\to \pi^{-1}(V)$ is said to be a
	\emph{normal map} or a map of \emph{normal coordinates} if it
        satisfies the following conditions:
	\begin{itemize}
		\item $\forall (z,v)\in U\times \S^1,\,\forall \theta \in \R,\,\rho(z,ve^{i\theta})=r_{\theta}\rho(z,v)$
		\item Identifying $\C^n$ with $T_{P_0}M$ as previously, one has: $$\forall (z,v)\in U\times \S^1,\,\pi(\rho(z,v))=\exp_M(z).$$ 
	\end{itemize}
\end{defn}

\begin{rem}
	The choice of a normal map around a point $P_0$ reflects the choice of an identification of $\C^n$ with $T_{P_0}(M)$ and a point over $P_0$ in $X$. Hence, if $\rho_1$ and $\rho_2$ are two normal maps around the same point $P_0$, then $\rho_1^{-1}\circ \rho_2\in U(n)\times U(1)$.
\end{rem}

Using Definition \ref{defn:normap}, one can compare, for $N$ large,
the Szeg\H{o} kernel $S_N$ with the flat case $\Pi_N$. For this, we push by
$\rho$ the Bargmann kernel and multiply by the correct factor in the
fibre to obtain an equivariant kernel on $X$:
$$\rho^*\Pi_N(\rho(z,\theta),\rho(w,\phi)):=e^{iN(\theta-\phi)}\Pi_N(z,w).$$

By convention, $\rho^*\Pi_N$ is zero outside $\pi^{-1}(V)^2$.

\begin{prop}[\cite{deleporte_low-energy_2016}]\label{prop:universal}
	Let $P_0\in M$, and $\rho$ a normal map around $P_0$. For every $\epsilon>0$ there exists $\delta 
	\in (0,1/2)$ and $C>0$ such that for every $N\in \N$, for every $u\in L^2(X)$, if the support of $u$ lies inside $\rho(B(0,N^{-\delta})\times \S^1)$, then $$\|(S_N-\rho^*\Pi_N)u\|_{L^2}<CN^{-\frac 12 + \epsilon}\|u\|_{L^2}.$$
\end{prop}

In a sense, Proposition \ref{prop:universal} states that the operator
$S_N$ asymptotically looks like $\Pi_N$ on small scales. The proof can
be found in a previous paper \cite{deleporte_low-energy_2016}, as a
consequence of previously known results on the asymptotical behaviour
of the Schwartz kernel of $S_N$ near the diagonal set
\cite{shiffman_asymptotics_2002,ma_holomorphic_2007,charles_berezin-toeplitz_2003}.

\subsection{Toeplitz operators}
\begin{defn}\label{def:toeplitz}
	Let $M$ be a K\"ahler manifold, with equivariant Szeg\H{o}
        projectors $(S_N)_{N\geq 1}$. 
	Let $h\in C^{\infty}(M)$ be a smooth function on $M$.
        For all $N\geq 1$, the Toeplitz operator $T_N(h):H_N(X)\to H_N(X)$ associated
        with the symbol $h$ is defined as $$T_N(h)=S_NhS_N.$$
      \end{defn}
      In this work we investigate the spectral properties of the operators
$T_N(f)$, for fixed $f$ and $N\to +\infty$.
\begin{ex}[Spin operators]
  Let us continue from Example \ref{ex:CP1}. The sphere $\S^2$ is
  naturally a submanifold of $\R^3$; as such, there are three
  coordinate functions $(x,y,z):\S^2\mapsto \R^3$. They are closed
  under Poisson brackets: one has $\{x,y\}=z$ and two similar identities 
  by cyclic permutation.

  In the Hilbert basis given by the normalized monomials, the associated Toeplitz operators $T_N(x),T_N(y),T_N(z)$ are, up to
  a factor $\frac{N}{N+2}$, the usual spin matrices with spin $\frac N2$.
\end{ex}

\subsubsection{Toeplitz operators on $\C^n$ and the Melin value}

The manifold $\C^n$ is not compact. Let us release the condition that
the symbol is bounded. This defines Toeplitz operators as unbounded
operators on $B_N$.

Toeplitz operators on $\C^n$ whose symbols are semipositive definite
quadratic forms play an important role in this work.
If $Q$ is a quadratic form on $\R^{2n}$ identified with $\C^n$, then
$T_N(Q)$ is essentially self-adjoint. This operator is related
to the Weyl quantization $Op^{\hbar}_W(Q)$ with semi-classical parameter
$\hbar=N^{-1}$. In fact, $T_{N}(Q)$ is conjugated, via the
Bargmann transform $\mathcal{B}_N$ \cite{bargmann_hilbert_1961}, with the operator
\[
  Op^{N^{-1}}_W(Q)+\cfrac{N^{-1}}{4}\tr(Q).
  \]

If $Q$ is semi-definite positive, then it takes non-negative values as
a function on $\R^{2n}$, hence $T_N(Q)\geq 0$ for all $N\geq 0$ since,
for $u\in B_N$, one has
\[
  \langle u,\Pi_NQ\Pi_Nu\rangle=\langle u,Qu\rangle\geq 0.
  \]
The
infimum of the spectrum of $T_N(Q)$ is of utmost interest, since it
leads to the notion of Melin value. As $Q$ is
$2$-homogeneous, and the Bargmann spaces are identified with
each other through a scaling, one has $T_N(Q)\sim N^{-1}T_1(Q)$, and
in particular 
\[
\inf( \Sp( T_N(Q)))=N^{-1}\inf (\Sp (T_1(Q))).
\]

\begin{defn}\label{def:charac}
	Let $Q$ be a semi-definite positive quadratic form on $\R^{2n}$, identified with $\C^n$.
	
	We denote by $\mu(Q)$ the \emph{Melin value} of $Q$, defined
        by 
\[
\mu(Q):=\inf\left(\Sp(T_1(Q))\right).
\]
\end{defn}

Given $Q\geq 0$, how can one compute $\mu(Q)$? As stated above, it
depends first on the trace of $Q$ (which is easy to compute), and
second on the infimum of the spectrum of $Op_W^1(Q)$. This second part
is invariant through a symplectic change of variables, and the problem
reduces to a \emph{symplectic diagonalisation} of $Q$ (see Proposition
\ref{prop:quadred}). As an example:

\begin{example}
  Let $\alpha,\beta\geq 0$. Then
\[
\mu\left((x,y)\mapsto \alpha x^2+\beta y^2\right)=\frac
14(2\sqrt{\alpha\beta}+\alpha+\beta).
\]
\end{example}
The function $\mu$ itself is \emph{not} invariant under
symplectomorphisms (for example, in the previous example it does not only depend on
$\alpha\beta$). However, it is invariant under unitary changes of
variables.

The regularity of the map $\mu$ will be useful in the
proof of Theorem A:
\begin{prop}[\cite{melin_lower_1971}]The function $Q\mapsto \mu(Q)$ is
  H\"older continuous with exponent $\frac{1}{2n}$ on the set of
  semi-definite positive quadratic forms of dimension $2n$.  
\end{prop}

If $Q$ is definite positive, then $T_N(Q)$ has compact
resolvent, and the first eigenvalue is simple.

\subsubsection{Toeplitz operators on compact manifolds}
When the base manifold $M$ is compact and $h$ is real-valued, for
fixed $N$ the operator $T_N(h)$ is a symmetric operator on a
finite-dimensional space. In this setting, we will speak freely about
eigenvalues and eigenfunctions of Toeplitz operators.

The composition of two Toeplitz operators can be
written, in the general case, as a formal series of Toeplitz
operators \cite{charles_berezin-toeplitz_2003}, that
is:$$T_N(f)T_N(g)=T_N(fg)+N^{-1}T_N(C_1(f,g))+N^{-2}T_N(C_2(f,g))+\ldots.$$

The composition properties of formal
series of Toeplitz operators lead to the following property, which
appears in previous work \cite{deleporte_low-energy_2016}, and which is an
important first step towards the study of the low-lying eigenvalues.

\begin{prop}\label{prop:preloc}
	Let $M$ be a compact K\"ahler manifold and $h$ a real
        non-negative smooth function on $M$.

Let $u=(u_N)_{N\geq 1}$ be a sequence of unit elements in the Hardy spaces such that, for every $N$, one has $$T_N(h)u_N=\lambda_Nu_N,$$ with $\lambda_N=O(N^{-1})$.
	
	Then the sequence $u$ localises on $\{h=0\}$. More
        precisely, for every $\epsilon>0$, if $$Z_N=\{m\in M, h(m)\geq
        N^{-1+\epsilon}\},$$ one has, as $N\to
        +\infty$, $$\int_{\pi^{-1}(Z_N)}|u(x)|^2\dd Vol=O(N^{-\infty}).$$
\end{prop}
 On
a point where $h$ is minimal, one can pull-back Definition
\ref{def:charac} by normal coordinates of Definition \ref{defn:normap}:

\begin{defn}\label{def:mu}
	Let $h\in \C^{\infty}(M,\R^+)$. Let $P_0\in M$ be such that
        $h(P_0)=0$. Let $\rho$ be a normal map around $P_0$; the function
        $h\circ \rho$ is well-defined and non-negative on a
        neighbourhood of $0$ in $\C^n$, and the image of $0$ is
        $0$. Hence, there exists a semi-definite positive quadratic
        form $Q$ such that
\[
h\circ \rho(x)=Q(x)+O(|x|^3).
\]

        We define the Melin value $\mu(P_0)$ as $\mu(Q)$.	
\end{defn}
\begin{rem}
  A different choice of normal coordinates corresponds to a $U(n)$
  change of variables for $Q$, under which $\mu$ is invariant. Hence, 
  $\mu(P_0)$ does not depend on the choice of normal coordinates.

  The function $P_0\mapsto \mu(P_0)$ is $\frac 1{2n}$-H\"older continuous as a composition
  of the smooth function $P_0\mapsto Q$ and the H\"older continuous
  function $Q\mapsto \mu$.
\end{rem}

\subsection{Quantum maps}
To a (local) symplectomorphism between K\"ahler manifolds, one can
associate an almost unitary (local) transformation on the Hardy
spaces, such that, at first order, the Toeplitz quantizations on both
sides are related by the symplectic change of variables in the symbols
\cite{charles_semi-classical_2007}:
\begin{prop}
  \label{prop:qmaps}
  Let $\sigma:(M,x)\mapsto (N,y)$ be a local symplectomorphism between
  two quantizable K\"ahler manifolds. 

Let $U$ be a small open set around $x$.
Then there exists, for every $N$, a linear map
$\mathfrak{S}_{N}:H_N(M,L)\mapsto H_N(N,K)$ and a sequence of
differential operators $(L_j)_{j\geq 1}$,
  such that, for any sequence $(u_N)_{N\geq 1}$ of sections which are
  $O(N^{-\infty})$ outside of $U$, and for any symbol $a\in
  C^{\infty}(N)$, one has:
  \begin{align*}
    \|\mathfrak{S}_Nu_N\|_{L^2}&=\|u_N\|_{L^2}+O(N^{-\infty})\\
    \mathfrak{S}_N^{-1}T_N(a)\mathfrak{S}_Nu_N&= T_N\left(a\circ
                                                \sigma +
                                                \sum_{k=1}^{\infty}N^{-i}L_j(a\circ
                                                \sigma)\right)u_N+O(N^{-\infty}).
  \end{align*}
  Moreover, for every $j\geq 1$, the differential operator $L_j$ is of
  degree $2j$.
\end{prop}

As a preliminary lemma for Sections 5 and 6, let us show that quantum maps preserve
concentration speed:
\begin{lem}\label{lem:qmaps-speed}
  Let $\sigma:(M,x)\mapsto (N,y)$ a local symplectomorphism between
  two quantizable K\"ahler manifolds.

  Let $0<\delta<\frac 12$ and let $(u_N)_{N\in \N}$ a sequence of unit elements in the Hardy spaces
  $H_N(M)$ such that
  \[\int_{\big\{\dist(\pi(y),x)\leq N^{-\frac 12 +
            \delta}\big\}}|u_N(y)|^2=O(N^{-\infty}).\]

  Then
  \[\int_{\big\{\dist(\pi'(y),\sigma(x))\leq N^{-\frac 12 +
            \delta}\big\}}|\mathfrak{S}u_N(y)|^2=O(N^{-\infty}).\]
    \end{lem}
    \begin{proof}
      Let us observe that the condition on $(u_N)$ is equivalent to
      the following: for every $k\in \N$, there exists $C_k>0$ such that
      \[\langle u_N,T_N(\dist(\cdot,x)^{2k})u_N\rangle \leq
        C_kN^{-k(1+2\delta)}.\]

      Let us prove, by induction on $k$, the estimate
      \[\langle \mathfrak{S}_Nu_N,T_N(\dist(\cdot,\sigma(x))^{2k})\mathfrak{S}_Nu_N\rangle \leq
        \tilde{C}_kN^{-k(1+2\delta)}.\] The case $k=0$ is clearly true
      since $\mathfrak{S}_N$ is an almost unitary operator when acting
      on functions localised near $x$.

      Let us now apply Proposition \ref{prop:qmaps} with
      $a=\dist(\cdot,x)^{2k}$, stopping the expansion at order $k$.

      For $j\leq k$, the error terms are controlled:
      \[\left|N^{-j}L_j(a\circ \sigma)\right|\leq
        N^{-j}C_{j,k}\dist(\cdot,\sigma(x))^{2(k-j)}\]

      Hence, by induction,
      \begin{multline*}
      \langle
        \mathfrak{S}_Nu_N,T_N(\dist(\cdot,\sigma(x))^{2k})\mathfrak{S}_Nu_N\rangle
        \\\leq \sum_{j=0}^{2k}C_{j,k}\langle
        u_N,T_N(\dist(\cdot,x)^{2(k-j)})u_N\rangle+O(N^{-k})=O(N^{k(-1+2\delta)}).
        \end{multline*}
      This ends the proof.
    \end{proof}


\section{Resolvent estimates}
We begin this section with a technical Lemma, which associates to a
function on a Riemannian manifold a covering of the manifold with
small open sets such that, on the intersections of the open sets, the
function is not relatively larger than elsewhere.

\subsection{A cutting lemma}
\begin{lem}\label{lem:cutting}
	Let $Y$ be a compact Riemannian manifold. There exist two
        constants $C>0$ and $a_0>0$ such that, for every positive integrable function $f$ on $Y$, for every $0<a<a_0$ and $t\in (0,1)$,
	there exists a finite family $(U_j)_{j\in J}$ of open subsets
        covering $Y$ with the following properties:
	\begin{align*}
		\forall j\in J,\, \diam(U_j)<a\\
		\forall j\in J,\, \dist\left(Y\setminus U_j,Y\setminus \bigcup_{i\neq j}U_i\right)\geq ta\\
		\sum_{i\neq j}\int_{U_i\cap U_j}f \leq Ct\int_Yf.
	\end{align*}
\end{lem}
\begin{preuve}
	Let $m\in \N$ be such that there exists a smooth embedding of differential manifolds from $Y$ to $\R^m$, and let $\Phi$ be such an embedding. $\Phi$ may not preserve the Riemannian structure, so let $c_1$ be such that, for any $\xi\in TY$, one has $$ c_1\|\Phi^*\xi\|\leq \|\xi\|.$$
	
	We now let $L>0$ be such that any hypercube $H$ in $\R^m$ of side
        $2/L$ is such that $\diam (\Phi^{-1}(H))<a$.  Since $\Phi^{-1}$ is
        uniformly Lipschitz continuous, then if $a$ is small enough one has $aL\leq C_1$
        for some $C_1$ depending only on $Y$.
	
        We then prove the claim with
        $C=\frac{2mC_1}{c_1}$.
	
	Let $1\leq k\leq m$, and let $\Phi_k$ denote the k-th
        component of $\Phi$. The function $\Phi_k$ is continuous from
        $Y$ onto a segment of $\R$. Without loss of generality this
        segment is $[0,1]$. Let $g_k$ denote the integral of $f$ along
        the level sets of $\Phi_k$.
        The function $g_k$ is a positive integrable function on $[0,1]$. Let $t'>0$ be the inverse of an integer, and $0\leq \ell \leq L-1$. In the interval $[\ell/L,(\ell+1)/L]$, there exists a subinterval $I$, of length $t'/L$, such that 
	\begin{equation}\label{eq:proof3.1}\int_I g_k \leq t'\int_{\ell/L}^{(\ell+1)/L}g_k.\end{equation} Indeed, one can cut the interval $[\ell/L,(\ell+1)/L]$ into $1/t'$ intervals of size $t'/L$. If none of these intervals was verifying (\ref{eq:proof3.1}), then the total integral would be strictly greater than itself.
	
	Let $x_{k,\ell}$ denote the centre of such an interval. Then, let
	\begin{align*}
		V_{k,0}&=\left[0,x_{k,0}+\frac{t'}{2L}\right)\\
		V_{k,\ell}&=\left(x_{k,\ell-1}-\frac{t'}{2L},x_{k,\ell}+\frac{t'}{2L}\right)\text{ for } 1\leq \ell \leq L\\
		V_{k,L+1}&=\left(x_{k,L}-\frac{t'}{2L},1\right].
	\end{align*}
	Each open set $V_{k,l}$ has a length smaller than $2/L$. The
        overlap of two consecutive sets has a length $t'/L$, and the
        sum over $k$ of the integrals on the overlaps is less than $t'\int_0^1g_k=t'\int_Yf$.
	
	Now let $\nu$ denote a polyindex $(\nu_k)_{1\leq k \leq m}$,
        with $\nu_k\leq L+1$ for every
        $k$. Define $$U_{\nu}=\Phi^{-1}\left(V_{1,\nu_1}\times
          V_{2,\nu_2}\times \ldots \times V_{m,\nu_m}\right).$$
        Then the family $(U_{\nu})_{\nu}$ covers $Y$. For every
        polyindex $\nu$ one has $\diam U_{\nu}\leq a$ since
        $U_{\nu}$ is the pull-back of an open set contained in a hypercube of side $2/L$. Moreover, one has
	$$\dist\left(Y\setminus U_{\nu},Y\setminus \bigcup_{\nu'\neq \nu}U_{\nu'}\right)\geq \frac{c_1t'}{L}.$$
	
	To conclude, observe that $$\sum_{\nu\neq \nu'}\int_{U_{\nu}\cap U_{\nu'}}f = \sum_{k=1}^m\sum_{\ell=0}^{L}\int_{V_{k,\ell}\cap V_{k,\ell+1}}g_k \leq mt'\int_Yf.$$
	
	It only remains to choose $t'$ conveniently. The fraction $t\frac{aL}{c_1}$ may not be the inverse of an integer; however the inverse of some integer lies in $[\frac{aL}{2c_1},\frac{aL}{c_1}]$. This allow us to conclude.
\end{preuve}

\begin{rem}
	In the previous Lemma, the number of elements of $J$ is bounded by a polynomial in $a$ that depends only on the geometry of $Y$.
\end{rem}
\subsection{Melin estimate}
The following Proposition is a Berezin-Toeplitz version of the well-known Melin
estimate for pseudodifferential estimates. It requires a weak
condition on the speed of growth of the symbol near its zero set.
\begin{prop}[Melin estimate]\label{prop:Melin}
Let $h\in C^{\infty}(M,\R^+)$ with $\min(h)=0$. Let
$$\mu_{\min}=\min_{h(x)=0}(\mu(x)).$$

Suppose there exist $C>0$ and $\alpha>0$ such that, for every $t\geq 0$,
one has
\[
\dist_H\left(\{h\leq t\},\{h=0\}\right)\leq C t^{\alpha}.
\]

Then there exist $\varepsilon>0$, $N_0$ and $C'>0$ such that, for every $N\geq
N_0$, one has
$$\min \Sp(T_N(h))\geq \mu_{\min}N^{-1}-C'N^{-1-\varepsilon}.$$
\end{prop}
\begin{proof}
We begin with a local result: There exist $\delta_0,
\delta_1,\varepsilon$ small
enough and $N_0$ such that, for $N\geq N_0$, for every $x\in M$ with $h(x)<N^{-1+\delta_1}$, for
every $u$ supported on $B(x,N^{-\frac 12 + \delta_0})\times \S^1$, one has
$$\langle S_Nu,hS_Nu\rangle \geq
(\mu_{\min}N^{-1}-CN^{-1-\varepsilon})\|u\|^2.$$

Indeed, the following holds by hypothesis:
\[
\dist(x,\{h=0\})\leq C N^{\alpha(-1+\delta_1)}.
\]
In particular,
\[
\Hess(h)(x)\geq -CN^{\alpha(-1+\delta_1)}.
\]
The following perturbation of $h$ is convex on $B(x,N^{-\frac 12 + \delta_0})$: 
\[
\tilde{h}:y\mapsto h(y) +
CN^{\max(\alpha(-1+\delta_1),-1/2+\delta_0)} \dist(y,x)^2.
\]

If now $u \in L^2(X)$ is supported on $B(x,N^{-\frac 12 +
  \delta_0})\times \S^1$, one has
\[
\left|\langle S_Nu,(h -\tilde{h})S_Nu\rangle \right| \leq CN^{-1+2\delta_0+\max(\alpha(-1+\delta_1),-1/2+\delta_0)}.
\]

As $\sqrt{h}$ is Lipschitz-continuous, one has
\[
\sup\left(\sqrt{h(y)},\dist(y,x)<2N^{-\frac
  12+\delta_0}\right)<CN^{\frac{-1+\delta_1}{2}}+CN^{-\frac 12 + \delta_0}.
\]
Hence,
\[
\sup\left(h(y),\dist(x,y)<2N^{-\frac
  12+\delta_0}\right)<CN^{-1+\max(\delta_1,2\delta_0)}.
\]

  Recall from Proposition \ref{prop:universal} that, for $\delta$ small enough, for every
  $x\in M$ with associated normal map $\rho$, for every $u$ with
  support inside $\rho(B(0,N^{-\frac 12 +\delta})\times \S^1)$, one has
  $$\|(S_N-\rho^*\Pi_N)u\|_{L^2}<CN^{-\frac 14}.$$
  
Hence, if $\delta_0<\delta$, then, for $N$ large enough,
\begin{align*}
\left|\langle(S_N-\Pi_N^*)u,\tilde{h} S_Nu \rangle\right| &\leq
CN^{-\frac 14} N^{-1+\max(\delta_1,2\delta_0)}\\
\left|\langle\Pi_N^*u,\tilde{h} (S_N-\Pi_N^*)u \rangle\right| &\leq
CN^{-\frac 14} N^{-1+\max(\delta_1,2\delta_0)}.
\end{align*}

If the function $\tilde{h}$ reaches its minimum on $B(x,N^{-\frac
  12+\delta_0})$ at an interior point $x'$ and if $Q$ denotes half of the Hessian matrix of $\tilde{h}$ at $x'$,
then
\[
\left|\langle \Pi_N u_*,\tilde{h}_*-Q,\Pi_Nu_*\rangle\right|\leq
CN^{-\frac 32 +3\delta_0}.
\]
Here, the subscript $_*$ denotes the pull-back by the normal
map.

If $\tilde{h}$ reaches its minimum at a boundary point $x'$, then if
$L$ denotes the differential of $\tilde{h}$ at $x'$ one has, by
convexity, for all $y\in B(x,N^{-\frac 12 + \delta_0})$,
\[
  L(y-x')\geq 0.
\]
In particular,
\[
  \langle \Pi_N^*u, \tilde{h}\Pi_N^*u\rangle \geq \langle
  \Pi_N^*u,(\tilde{h}-L)\Pi_N^*u\rangle.
\]

Then $y\mapsto \tilde{h}(y)-L(y,x')$ has a critical point at $x'$. If $Q$
denotes again half of the Hessian matrix of $\tilde{h}$ at $x'$, then
\[
\left|\langle \Pi_N u_*,\tilde{h}_*-L-Q,\Pi_Nu_*\rangle\right|\leq
CN^{-\frac 32 +3\delta_0}.
\]

Since in any case $\dist(x',\{h=0\})\leq C N^{\max(\alpha(-1+\delta_1),-\frac 12 +\delta_0)}$, the matrix $Q$
is $C N^{\varepsilon}$-close to half of the Hessian matrix of $h$ at a
zero point (recall we only added $CN^{\alpha(-1+\delta_1)}I$ to the
Hessian matrix of $h$ at $x$.)

The Melin value $\mu$ is H\"older-continuous with exponent $(2n)^{-1}$
on the set of
semi-positive quadratic forms \cite{melin_lower_1971}, hence
\[
\mu(Q)\leq \mu_{\min}+CN^{\varepsilon}.
\]
To conclude, for $\varepsilon$ small enough depending on
$M,\alpha,\delta_0,\delta_1$, one has
\[
\langle S_Nu,hS_Nu\rangle\geq N^{-1}\mu_{\min}+CN^{-1+\varepsilon}.
\]

From this local estimate, we deduce a global estimate using Lemma \ref{lem:cutting} proved previously, and general localisation
estimates proved in \cite{deleporte_low-energy_2016}.

Indeed, let $(u_N)_{N\geq 1}$ be a sequence of normalised eigenfunctions for
$T_N(h)$ with minimal eigenvalue. Either the associated sequence of eigenvalues is
not $O(N^{-1})$, in which case the proposition clearly holds, or it
is, in which case, by Proposition \ref{prop:preloc}, $u_N$ is $O(N^{-\infty})$ outside $\{h\leq
N^{-1+\delta_1}\}$ for every $\delta_1>0$.

We now invoke Lemma \ref{lem:cutting} with the following data:
\begin{itemize}
\item $Y=M.$
\item $f=|u_N|^2$.
\item $a=N^{-\frac 12 + \delta_0}.$
\item $t=N^{-\frac{\delta_0}{2}}.$
\end{itemize}

The Lemma yields a sequence of coverings $(U_{j,N})_{j\in J_N,N\in
  \N}$. The proof also yields a sequence of coverings by slightly
smaller open sets $(U'_{j,N})$, with
\begin{itemize}
\item $U'_{j,N}\subset U_{j,N}$.
\item $d(M\setminus U_{j,N},U'_{j,N})\geq\frac
  12N^{-\frac{1-\delta_{0}}{2}}$.
\end{itemize}

Let $(\chi_{j,N})_{j\in J_N,N\in \N}$ be a partition of unity associated
with $(U'_{j,N})_{j\in J_N,N\in \N}$.

Then, for every $N$, for every $j\neq k\in J_N$, the integral $$\langle
S_N\chi_{j,N}u_N,hS_N\chi_{k,N}u_N\rangle$$ is $O(N^{-\infty})$
outside $(U_{j,N}\cap U_{k,N})^3$. Moreover, $S_N\chi_{k,N}u_N$ is
$O(N^{-\infty})$ outside $\{h\geq N^{-1+\delta_1}\}$ (because of Proposition \ref{prop:preloc}). Hence, $$|\langle
S_N\chi_{j,N}u_N,hS_N\chi_{k,N}u_N\rangle| \leq
CN^{-1+\delta_1} \||u_N|^2\|_{L^1(U_{j,N}\cap U_{k,N})}+O(N^{-\infty}).$$

Hence, by Lemma \ref{lem:cutting}, for every $N$, one has
$$\sum_{j\neq k\in J_N}|\langle
\chi_{j,N}u_N,T_N(h)\chi_{k,N}u_N\rangle|\leq
CN^{-1+\delta_1}N^{-\frac{\delta_0}{2}}\|u_N\|_{L^2}^2+O(N^{-\infty}).$$
(Recall $|J_N|$ has polynomial growth in $N$.)

On the other hand, the following holds:
$$\sum_{j\in J_N}\langle \chi_{j,N}u_N,T_N(h)\chi_{j,N}u_N\rangle \geq
(\mu_{\min}N^{-1}-CN^{-1-\epsilon})\sum_{j\in
  J_N}\|S_N\chi_{j,N}u_N\|_{L^2}^2.$$

Since $S_N\chi_{j,N}u_N$ and $S_N\chi_{k,N}u_N$ are almost orthogonal for
$j\neq k$, one has $$\sum_{j\in
  J_N}\|S_N\chi_{j,N}u_N\|_{L^2}^2\geq (1-CN^{-\epsilon})\|u_N\|_{L^2}^2.$$

Then, choosing $\delta_1<\frac{\delta_0}{2}$ allows us to conclude:
$$\langle u_N,T_N(h)u_N\rangle \geq N^{-1}(\mu_{\min}-CN^{-\epsilon})\|u_N\|_{L^2}^2.$$
\end{proof}

Note that, in the last proof, it is essential that we know beforehand
that $u_N$ is $O(N^{-\infty})$ on $\{h\geq N^{-1+\delta}\}$ for
every $\delta>0$. This was achieved by picking $u_N$ as the unique
minimizer of $\langle u, T_N(h)u\rangle$ under $\|u\|=1$, in which
case $u_N$ is an eigenfunction of $T_N(h)$.

\begin{rem}
Proposition \ref{prop:Melin} only relies on elementary properties of the Szeg\H{o}
kernel and Toeplitz operators (that is, Propositions \ref{prop:ma4.1}
and \ref{prop:universal}). As such, it extends readily to more general
contexts of quantizations, such as Spin$^c$-Dirac
\cite{ma_holomorphic_2007} (up to a
modification in the definition of $\mu_{\min}$).

The condition of polynomial growth near zero is satisfied for every
analytic symbol, in particular, it is true for spin systems considered
in Section 8.
\end{rem}
\subsection{Pseudo-locality of the resolvent}
\begin{prop}
\label{prop:plr}
  Let $h$ and $\mu_{\min}$ be as in Proposition
  \ref{prop:Melin} and $\epsilon_0>0$ small enough, depending only on
  $M$. Then, for every $0\leq\epsilon<\epsilon_0$ and every $c>0$, the operator
  $T_N(h-N^{-1}\mu_{\min}+cN^{-1-\epsilon})$ is invertible (as a positive
  definite operator on a finite-dimensional space). Its inverse
  $R$ is
  pseudo-local: if $a$ and $b$ are smooth functions with $\supp(a)\cap
  \supp(b)=\emptyset$, then $$T_N(a)RT_N(b)=O_{L^2\to L^2}(N^{-\infty}).$$
\end{prop}
\begin{proof}
The proposition may be reformulated this way: if $U \subset \subset V$
are two open sets in $M$ and a sequence $(u_N)_{N\geq 1}$ of
normalised states is such that
$T_N(h-N^{-1}\mu_{\min}+cN^{-1-\epsilon})u_N=O_{L^2}(N^{-\infty})$ on $V$,
then we wish to prove that $u_N=O_{L^2}(N^{-\infty})$ on $U$. Here
$$\supp(a)\subset \subset U \subset \subset V \subset \subset
(M\setminus \supp(b)).$$

We first remark that, for every $\delta$, and for every
$U\subset\subset V_1\subset
\subset V$, the following holds:
$$\int_{V}\overline{u}\,T_N(h)u\geq CN^{-1+\delta}\int_{V_1\cap\{h\geq
  N^{-1-\delta}\}}|u|^2.$$ Hence, $u$ is $O(N^{-\infty})$ on $V_1\cap\{h\geq
N^{-1-\delta}\}$ for every $\delta$. 

We are now able to repeat the global part of the proof of Proposition \ref{prop:Melin} by
cutting a neighbourhood of $U$ into small pieces, hence the claim.
\end{proof}


\section{Proof of Theorem A}
\subsection{Estimate of the first eigenvalue}
\begin{prop}\label{prop:upest}
  Let $h\geq 0$ with $\min(h)=0$ and let $\mu_{\min}$ be as in Proposition \ref{prop:Melin}. Then
  there exists $\epsilon>0$ such that
$$\inf Sp(T_N(h)) \leq N^{-1}\mu_{\min} + N^{-1-\epsilon}.$$
\end{prop}
\begin{proof}
  Let $P_0\in M$ achieve the minimal value $\mu_{\min}$, let $\rho$ be
  a
  normal map around $P_0$, and let $\delta>0$ and $C>0$ be such that, for every $N$, for
  every $u$ supported on $B(P_0,N^{-\frac 12+\delta})\times \S^1$, one has
  $\|(S_N-\rho^*\Pi_N)u\|\leq C N^{-\frac 14}$. Without loss of generality
  $\delta<\frac 18$.

  Pick $\alpha \geq 2 \delta$, and let $Q$ denote half of the Hessian of $h$
  at $P_0$. Then, since the function $Q\mapsto \mu(Q)$ is H\"older
  continuous with exponent $\frac{1}{2n}$ \cite{melin_lower_1971}, one has
  $$\mu(Q+N^{-\alpha}|\cdot|^2)\leq \mu(Q)+CN^{-\frac{\alpha}{2n}}.$$
  Let $v_N$ denote a normalised ground state of
  $T_N(Q+N^{-\alpha}|\cdot|^2)$, then $v_N$ is $O(N^{-\infty})$
  outside $B(0,N^{-\frac 12 + \delta})$.
  Then $\langle \rho^*v_N,T_N(h)\rho^*v_N\rangle =
  \mu(Q+N^{-\alpha}|\cdot|^2)+O(N^{-\frac 54 +
    2\delta}) \leq \mu_{\min}N^{-1}+O(N^{-1-\epsilon})$ for some $\epsilon>0$.
\end{proof}

\subsection{End of the proof}
We can now conclude the proof of Theorem A. Let $a\in C^{\infty}(M)$
be supported away from the set of points achieving $\mu_{\min}$. Let
$\tilde{h}\in C^{\infty}(M)$ be such that $\tilde{h}=h$ on the support
of $a$, and such that $\mu_{\min}(\tilde{h})>\mu_{\min}(h)$.
Then $T_N(\tilde{h}-N^{-1}\mu_{\min}(h))$ is invertible because of
the Melin estimate of Proposition \ref{prop:Melin}. Its inverse $R$ is
pseudolocal, with norm $O(N)$; in particular, for every integer
$k$,
\[
  T_N(a)=T_N(h-N^{-1}\mu_{\min}(h))^kR^kT_N(a)+O(N^{-\infty}).
\]
If
$u_N$ is a sequence of unit ground states of $T_N(h)$, then by
Propositions \ref{prop:Melin} and \ref{prop:upest} there holds
\[
  |\langle u_N,T_N(h-N^{-1}\mu_{\min}(h)) u_N\rangle|\leq
  CN^{1+\epsilon}.
\]
Hence, $$\langle u_N,T_N(a)u\rangle \leq C^kN^{-k-k\epsilon}\langle
u,R^kT_N(a)u_N.\rangle$$ In
particular, for every integer k, one has $$\langle
u_N,T_N(a)u_N\rangle=O(N^{-k\epsilon}),$$ which concludes the proof of
Theorem A. 

\section{The regular case}
In this Section we prove Theorem B, and establish the necessary
material for the Weyl asymptotics of Section 7.

We first study a problem of symplectic geometry, which consists in
finding a normal form for a non-negative function vanishing at order
$2$ on an isotropic submanifold. Then, we apply a Quantum Map to find
an expansion of the first eigenvalue and eigenfunction.
\subsection{A convenient chart}
We begin with the following fact:
\begin{prop}\label{prop:quadred}
  Let us endow $\R^{2n}$ with the canonical symplectic structure.

  Let $Q:\R^d\mapsto S_{2n}^+(\R)$ be a smooth $d$-parameter family
  of semi-positive quadratic forms. Suppose $rank\ Q=r$ is a constant
  function and suppose that, for every $t\in \R^d$, the space
  $\ker(Q(t))$ is isotropic. In particular, $\ker Q$ is a smooth map into the set of isotropic
  subspaces of $\R^{2n}$.

Then there is a smooth family $(e_1,f_1,\ldots,e_n,f_n):\R^d\mapsto (\R^n)^n$ of
symplectic bases, and smooth functions
$\lambda_{i}:\R^d\mapsto \R^*_+$, $r+1\leq i \leq n$, such that
$$Q\left(\sum_{i=1}^n q_ie_i+p_if_i\right)=\sum_{i=1}^r
p_i^2+\sum_{i=r+1}^{n}\lambda_i(p_i^2+q_i^2).$$

In particular, under the conditions above, the function $\mu(Q)$ is smooth.
\end{prop}
In the study of the Hamiltonian dynamics related to $Q$,
the vectors $f_i,1\leq i \leq r$ are called \emph{slow
  modes}. They correspond to the motion of a free particle. The vectors $(e_i,f_i),r+1\leq i \leq n$ are
called \emph{fast modes} and correspond to harmonic oscillations; the
associated values $\lambda_i,r+1\leq i \leq n$ are the
\emph{symplectic eigenvalues} of $Q$.

Here, the \emph{zero modes} (that is, the kernel of $Q$) are supposed
to form an isotropic subspace; this can be generalised, but
$\ker Q$ must have a constant symplectic rank in order for such a
result to apply.

\begin{proof}
  Let us construct a symplectic basis
  $(e_1,\ldots,e_n,f_1,\ldots,f_n)$ of $\R^{2n}$,
 depending
  smoothly on the parameters, on which the quadratic form $Q$ is
  diagonal, with the desired diagonal terms. We proceed by partial
  induction: if $Q$ is degenerate, we construct the first pair
  $(e_1,f_1)$ with $e_1\in \Ker Q$, hence the reduction to $Q'$ on
  $\R^{2(n-1)}$ with $\dim \Ker Q'=r-1$. If $Q$ is non-degenerate, we
  use a standard construction of the full basis in one step.

  Suppose $r>0$.
  Pick $e_1\in \ker Q$ smoothly depending on the parameters. The quadratic form $Q$
  is degenerate, but it is a
  well-known fact that it has no co-isotropic subspaces: if a subspace
  $F$ is such that \[\{e\in \R^{2n},\,\forall f\in
  F,\,Q(e+f)=Q(e)+Q(f)\}\subset F,\] then $F=\R^{2n}$.
  
  Hence, with $F=\{z\in \R^{2n},\langle z,Je_1\rangle=0\}$ denoting the
  symplectic orthogonal of $e_1$, there exists $f_1$ such that:
\begin{align*}
  \langle e_1,Jf_1\rangle &=1\\
\forall z\in F,\,Q(z+f_1)&=Q(z)+Q(f_1).
\end{align*}
The vector $f_1$ again depends smoothly on the parameters. As
$\lambda=Q(f_1)$ is far from zero on compact sets (recall that $\ker Q$ is
a continuous family of isotropic subspaces), changing $e_1$ into
$\sqrt{\lambda}e_1$ and $f_1$ into $f_1/\sqrt{\lambda}$ yields two
smooth vectors with the supplementary condition that $Q(f_1)=1$.

If one can find a smooth symplectic basis
$(e_2,\ldots,e_n,f_2,\ldots,f_n)$ of the symplectic orthogonal of
$\Span(e_1,f_1)$, which diagonalises the restriction of $Q$ with
diagonal values as above, then
completing this basis with $e_1$ and $f_1$ concludes the proof.

If $r=0$, let $M$ be the matrix of $Q$ in the (symplectic) canonical basis. Then $M^{\frac 12}$ is a smooth family of symmetric
matrices, so that $M^{\frac 12}JM^{\frac 12}$ is a smooth family of
antisymmetric matrices, where $J$ is the matrix of the standard
symplectic form in the canonical basis. Hence, there is a smooth family $U$ of
orthogonal matrices, and a smooth family $D$ of positive diagonal
matrices, such that $$U^TM^{\frac 12}JM^{\frac
  12}U=\begin{pmatrix}0&D\\-D&0\end{pmatrix}.$$

In particular, with $A=\begin{pmatrix}D^{\frac 12}&0\\0&D^{\frac 12}\end{pmatrix}$,
  one has
$$(AU^TM^{\frac 12}J)M(-JM^{\frac
  12}UA)=\begin{pmatrix}D&0\\0&D\end{pmatrix},$$
and $$(AU^TM^{\frac 12}J)J(-JM^{\frac 12}UA)=J.$$

Hence, the desired symplectic matrix is $AU^TM^{\frac 12}J$, which
depends smoothly on the parameters. This concludes the proof.
\end{proof}

Recall the following well-known application of Moser's principle:
\begin{prop}\label{prop:moser}
  Let $S$ be a symplectic manifold and $Z\subset S$ be a smooth
  $d$-dimensional submanifold
  of constant symplectic rank. Then, in a neighbourhood (in $S$) of any point
  in $Z$, there is a symplectomorphism $\rho$ onto a neighbourhood
   $0$ of
  $\R^{2n}$, such that $\rho(Z)$ is a piece of linear subspace.
\end{prop}

Using the two previous Propositions, we will prove the normal form for
miniwells on isotropic submanifolds:
\begin{prop}\label{prop:coords}
  Let $h$ be a smooth non-negative function on $M$, which vanishes on
  an isotropic manifold $Z$ of dimension $r$ with everywhere
  non-degenerate transverse Hessian.

Near any point of $Z$, there is a symplectomorphism $\rho$ into
$\R^{2r}_{q,p}\times \R^{2(n-r)}_{x,\xi}$, a smooth function
$Q_S$ from $\R^r$ into the set of positive quadratic forms of
dimension $r$, and $n-r$ smooth positive functions $(\lambda_i)_{1\leq
  i \leq n-r}$ such that:
$$h\circ
\rho^{-1}=\sum_{i=1}^{n-r}\lambda_i(q)(x_i^2+\xi_i^2)+Q_S(q)(p)+O_{(x,\xi,p)\to
  0}(|x|^3+|\xi|^3+|p|^3).$$
\end{prop}

In particular, $Z$ is mapped into
$\{(p,x,\xi)=0\}$.

\begin{proof}
  Let $P_0\in Z$, and let $U$ be a small neighbourhood of $P_0$ in $M$.
  Let us use Proposition \ref{prop:quadred} with set of parameters
  $Z\cap U$ and quadratic form $\Hess(h)$, which is semi-positive
  definite along $Z\cap U$, with kernel of constant symplectic rank.
  
This yields, at each point of $Z$ in a
neighbourhood of $P_0$, a family of $2n$ vector fields which form a
symplectic basis: $$\mathcal{B}=(Q_1,\ldots,Q_r,P_1,\ldots,P_r,X_1,\ldots,
X_{n-r},\Xi_1,\ldots,\Xi_{n-r}),$$ such that $\Span(Q_1,\ldots,Q_r)=TZ$.
In the general setting, this does not give a symplectic change of
variables under which the quadratic form is diagonal along the whole zero set (indeed,
$Q_1,\ldots,Q_r$ are prescribed by the $2n-r$ other vector fields, and
do not commute in general). However, one can separate the slow
variables and the fast variables (first step), then diagonalise the
fast variables (second step).

{\it First step}: Let us define the distribution $\mathcal{F}$ on
$Z\cap U$
as follows: for $x\in Z\cap U$,
\[
  \mathcal{F}_x=\Span(Q_1,\ldots,Q_r,P_1,\ldots,P_r)(x).
\]
Then $T(Z\cap U)\subset \mathcal{F}$. In particular, there is a piece
$S$ of symplectic submanifold of $M$, containing $Z\cap U$, and tangent to
$\mathcal{F}$ on $Z\cap U$.

Using Proposition \ref{prop:moser}, we let $\phi_0$ be a
symplectomorphism from a neighbourhood of $P_0$ in $M$ into a
neighbourhood of $0$ in $\R^{2r}\times \R^{2(n-r)}$, such that $S$ is
mapped into $\R^{2r}\times \{0\}$. Using Proposition \ref{prop:moser}
again, let $\phi_1$ be a symplectomorphism
on a neighbourhood of $0$ in $\R^{2r}$, that maps $\phi_0(Z)$ into
$\R^r\times \{0\}$. Then the map $\tilde{\phi}_1$ acting on
$\R^{2r}\times \R^{2(n-r)}$ by
$$\tilde{\phi_1}(p,q,x,\xi)=(\phi_1(p,q),x,\xi)$$ is a
symplectomorphism. We claim that $\rho=\tilde{\phi}_1\circ \phi_0$
separates the fast variables from the slow variables up to
$O((x,\xi,p)^3)$.

Indeed, consider $D\rho$ at a point of $Z$. Since $\rho$ sends $Z$
into $\R^r\times \{0\}$, and $S$ into $\R^{2r}\times \{0\}$, the matrix
of $D\rho$, from the basis $\mathcal{B}$ to the canonical basis, is of
the form:
$$D\rho=\begin{pmatrix}[cc|cc]A_{qq}&0&0&0\\A_{pq}&A_{pp}&0&0\\ \hline
  A_{xq}&A_{xp}&A_{xx}&A_{x\xi}\\A_{\xi
    q}&A_{\xi p}&A_{\xi x} & A_{\xi \xi}\end{pmatrix}.$$
Moreover, $D_\rho$ is symplectic, so that the bottom left part
vanishes. Hence, 
$$h\circ \rho^{-1}=Q_{F}(q)(x,\xi)+Q_S(q)(p)+O(|p|^3+|x|^3+|\xi|^3),$$ for
some quadratic forms $Q_F$ and $Q_S$. 

Since $h$ vanishes at order exactly $2$ on $Z$, the quadratic forms $Q_F$ and
$Q_S$ are positive definite.

{\it Second step}: It only remains to diagonalise $Q_F$ with a symplectomorphism. In
fact, this is possible without modifying $Q_S$. Indeed, let
$\phi:(\R^r,0)\mapsto Sp(2(n-r))$ be such that, for every $q$ near zero, the
matrix $\phi(q)$ realises a symplectic reduction of $Q_F(q)$, with
eigenvalues $\lambda_{1}(q),\ldots,\lambda_{n-r}(q)$. All of this
depends smoothly on $q$ by Proposition \ref{prop:quadred}. With $J$ the
standard complex structure matrix on $\R^{2(n-r)}$ and $\langle
\cdot,\cdot \rangle$ its standard Euclidian norm, we define, for
every $1\leq i \leq r$, the real function
$$f_i:(q,x,\xi)\mapsto \frac 12 \langle
(x,\xi),(\partial_{q_i}\phi(q)J\phi^t(q))(x,\xi)\rangle.$$

We then define $f:(\R^{2n-r},0)\to \R^r$ as the map with components $f_i$
in the canonical basis. Then a straightforward computation shows that the map
$$\Phi:(q,p,x,\xi)\mapsto(q,p+f,\phi(q)(x,\xi))$$ is a
symplectomorphism. As $f=O_{(x,\xi)\to 0}((x,\xi)^2)$, the 2-jet of
$h\circ \Phi$ at $(q,0,0,0)$ is the same as the 2-jet of
$h\circ((q,p,x,\xi)\mapsto (q,p,\phi(q)(x,\xi)))$,
i.e. $Q_S(q)(p)+\sum_{i=1}^{n-r}\lambda_i(q)(x_1^2+\xi_i^2)$. This
concludes the proof. 
\end{proof}
\begin{rem}
  We corrected the map $$(q,p,x,\xi)\mapsto
  (q,p,\phi_q(x,\xi))$$ into a symplectomorphism by only changing the
  second coordinate. This does not depend on the fact that $\phi_q$
  acts linearly but relies only on $\phi_q(0,0)=(0,0).$
\end{rem}
\subsection{Approximate first eigenfunction}\label{subsec:reg-first-eigenvalue}
Let us quantize, using
Proposition \ref{prop:qmaps}, the symplectic map of Proposition \ref{prop:coords}, and
conjugate with pseudodifferential operators:
\begin{defn}
  For any choice $\mathfrak{S}_N$ of quantization of the map $\rho$ of
  Proposition \ref{prop:coords}, the classical symbol
  $g_{\mathfrak{S}}\sim \sum N^{-i}g_i$ on a neighbourhood $U$ of $0$ in
  $\R^{2n}$ is defined as follows: for any sequence $(v_N)_{N\geq 1}$ with microsupport in
  a compact set of $U$, the following holds:
$$\mathcal{B}_N^{-1}\mathfrak{S}_N^{-1}T_N(h)\mathfrak{S}_N\mathcal{B}_Nv_N=Op_W^{N^{-1}}(g_{\mathfrak{S}})v_N
+ O(N^{-\infty}).$$
\end{defn}
In what follows, we choose an arbitrary quantum map $\mathfrak{S}_N$,
and we write $g$ instead of $g_{\mathfrak{S}}.$
The reason we use Weyl quantization in this subsection is because we will
rely heavily on squeezing operators. The computations are much easier
to follow for this formalism.

The principal and subprincipal symbols of $g$ are explicit at the points
of interest: $g_0=h\circ \rho$ by construction, and $g_1$ is prescibed on $\{g_0=0\}$ by the Melin
estimates for Weyl and Toeplitz quantizations:
\begin{prop}\label{prop:Melin-reg}
  For any $q$ close to $0$, one has $$g_1(q,0,0,0)=\frac 12 \sum_{i=1}^{n-r}\lambda_i(q)+\frac 14 \tr(Q_S(q)).$$
\end{prop}
\begin{proof}
  From the expression of $h\circ \rho$ in Proposition
  \ref{prop:coords}, one has
\[
\mu(\rho(q,0,0,0))=\sum_{i=1}^{n-r}\lambda_i(q)+\frac 14 \tr(Q_S(q)).
\]

If $h(x)=0$ and $\delta>0$ is small enough, the value $\mu(x)$ has the following variational characterisation:
\[
\mu(x)=\lim_{N\to +\infty}\left(N\inf\left(\int_M h|u|^2,u\in
    H_N(X),\int_{B(x,N^{-\frac 12+\delta})}|u|^2=1\right)\right)
\]

This variational problem can be read via the quantum map. If
\[
\int_{B(x,2N^{-\frac 12+\delta})}|u|^2=O(N^{-\infty}),
\]
then $\mathcal{B}_N^{-1}\mathfrak{S}_N^{-1}u$ microlocalises at speed
$N^{-\frac 12 + \delta}$ on $\rho^{-1}(x)$, and moreover,
\[
\int_M h|u|^2=\left\langle \mathcal{B}_N^{-1}\mathfrak{S}_N^{-1} u, Op^{N^{-1}}_W(g_0+N^{-1}g_1) \mathcal{B}_N^{-1}\mathfrak{S}_N^{-1}u
\right\rangle+O(N^{-2})\|u\|_2.
\] 

Now, if $x=\rho(q,0,0,0)$, the usual Melin estimate yields
\[
\lim_{N\to +\infty}\left(N\inf\left(\left\langle v
    Op_W^{N^{-1}}(g_0) v\right\rangle, \mathfrak{S}_N\mathcal{B}_N
    v\text{ as above}\right)\right)=\frac 12 \sum_{i=1}^{n-r}\lambda_i(q),
\]
hence, $g_1(\rho(q,0,0,0))$ contains all the defect between $\mu(\rho(q,0,0,0))$ and this estimate.
\end{proof}

\begin{rem}
  In general, the subprincipal symbol is not unique after application
  of a quantum map. Indeed, if $a$ is any smooth real-valued function on $M$ then
  $\exp(iT_N(a))$ is a unitary operator, and composing
  $\mathfrak{S}_N$ with this operator changes the subprincipal term.

Proposition \ref{prop:Melin-reg} shows that on the points where the principal symbol vanishes, the
subprincipal symbol is in fact rigid through any such transformations.
\end{rem}

Let us find a candidate for an approximate first eigenfunction:

\begin{prop}\label{prop:eigenf-reg}
  Suppose that the function $q\mapsto \mu\circ\rho(q,0,0,0)$ has a
  non-degenerate minimum at $q=0$. Let $\phi$ be the positive quadratic
  form such that $q\mapsto e^{-\phi(q)}$ is the ground state
  of the operator $$Q_S(0)(-i\nabla)+\frac 12
  \sum_{i,j=1}^rq_iq_j\cfrac{\partial^2}{\partial q_i\partial q_j}(\mu
  \circ \rho)(0,0,0,0),$$ with eigenvalue $\mu_2$.

  Then there exists a sequence of polynomials $(b_i)_{i\geq 1}$, and a
  sequence of real numbers $(\mu_i)_{i\geq 1}$, with
  \begin{align*}
    \mu_0&=\mu\circ \rho(q,0,0,0)\\
           \mu_1&=0
  \end{align*}and 
$\mu_2$
as previously, such
  that, for every $k$, $$f^k_N:(q,x)\mapsto N^{\frac n2-\frac r4}e^{-N\frac{x^2}{2}-\sqrt{N}\phi(q)}\left(1+\sum_{i=1}^{k}N^{-\frac
    i4}b_i(N^{\frac 14}q,N^{\frac 12}x)\right)$$ is an approximate
  eigenvector to $Op_W^{N^{-1}}(g)$, with
  eigenvalue $$\lambda^k_N=N^{-1}\sum_{i=0}^{k}N^{-\frac i4}\mu_i,$$ in the
  sense that, for every $K$ there exists $k$ such that
  $$\|Op_W^{N^{-1}}(g)f^k_N-\lambda^k_Nf^k_N\|_{L^2}=O(N^{-K}).$$
\end{prop}
This proposition provides an almost eigenfunction which we will show
to be associated to the lowest eigenvalue (see Proposition
\ref{prop:spgap-reg}). It is the main argument in the proof of Theorem
B; the concentration speed of this eigenfunction on zero, which is $N^{-\frac
  14}$, is the concentration speed of the lowest eigenvector of
$T_N(h)$ on the miniwell $P_0$, because of Proposition \ref{lem:qmaps-speed}.
\begin{proof}
  The proof proceeds by a squeezing of $Op_W^{N^{-1}}(g)$ by a
  factor $N^{\frac 14}$ along the $q$ variable.

  Let $$\tilde{g}_N=g(N^{-\frac 14}q,N^{-\frac 34}p,N^{-\frac
    12}x,N^{-\frac 12}\xi).$$

  Then $Op^{N^{-1}}_W(g_N)$ is conjugated with $Op^1_W(\tilde{g}_N)$
  through the unitary change of variables $u\mapsto N^{\frac n2-\frac r4}u(N^{-\frac 14}q,N^{-\frac 12}x).$

  Grouping terms in a Taylor expansion of $\tilde{g}_N$ yields
  $$\tilde{g}_N=N^{-1}\sum_{i=0}^{K}N^{-\frac
    i4}a_i(q,p,x,\xi)+O(N^{-\frac{K+5}{4}}),$$with first terms
\begin{align*}
  a_0=&g_1(0,0,0,0)+\sum_{i=1}^{n-r}\lambda_i(0)(x_i^2+\xi_i^2)\\
a_1=&q\cdot\nabla_q\left(g_1(\cdot,0,0,0)+\sum_{i=1}^{n-r}\lambda_i(\cdot)(x_i^2+\xi_i^2)\right)(0)\\
a_2=&Q_S(p)+\frac 12\Hess_q\left(g_1(\cdot,0,0,0)+\sum_{i=1}^{n-r}\lambda_i(\cdot)(x_i^2+\xi_i^2)\right)(0)(q)\\
      &+R_3(x,\xi)+L(x,\xi).
\end{align*}
Here $R_3$ is a homogeneous polynomial of degree $3$ and $L$ is a
linear form.

We further write $A_i=Op_W^{1}(a_i)$.

Recall $g_1(q,0,0,0)+\frac 12 \sum \lambda_i(q)=\mu\circ
\rho(q,0,0,0)$ and let $\phi$ be the positive quadratic form such that $e^{-\phi}$ is the ground
state (up to a positive factor) of $$Op_W^1\left(Q_S(p)+
\frac 12 \Hess(\mu\circ \rho)(0)(q)\right),$$ and
let $$u_0(q,x)=e^{-\frac{x^2}{2}-\phi(q)}.$$ We will provide a sequence
of almost eigenfunctions of $Op_W^{1}(\tilde{g}_N)$, of the form $$u_0(q,x)\left(1+\sum_{i=1}^{+\infty}N^{-\frac
    i4}b_i(q,x)\right),$$ with approximate
  eigenvalue $$N^{-1}\sum_{i=0}^{+\infty}N^{-\frac i4}\mu_i.$$ We
  proceed by perturbation of the dominant order $A_0$, which does not
  depend on $q$. 
  Our starting point is
  $$u_0=e^{-\frac {|x|^2}{2}-\phi(q)},\,\mu_0=\min \Sp(A_0)$$
  $$u_1=0,\,\mu_1=0.$$
  Indeed, one has $A_0u_0=\mu_0u_0$, and $A_1u_0=0$
  since $$\nabla\left(g_1(\cdot,0,0,0)+\frac 12 \sum_{i=1}^{n-r}\lambda_i(\cdot)\right)(0)=0,$$ so that $u_0$ is
  an approximate eigenvector for $Op^1_W(\tilde{g})$.

  Let us proceed by induction. Let $k\geq 1$ and suppose that we have
  already built $u_0,\ldots,u_k$ and $\mu_1,\ldots,\mu_k$ which solve
  the eigenvalue equation at order $k$; suppose further that there
  exists $C_{k+1}\in \R$ such that, for every $q\in \R^r$,

  $$\int_{\R^{n-r}}\overline{u_0}(x,q)\left(\sum_{i=1}^{k+1}[A_iu_{k+1-i}](q,x)-\sum_{i=1}^{k}\mu_iu_{k+1-i}(q,x)\right)\dd
  x = C_{k+1}|u_0(x,q)|^2.$$

  Then one can solve the equation
  $$(A_0-\mu_0)u_{k+1}+\cdots+(A_{k+1}-\mu_{k+1})u_0=0,$$
  up to a multiple of $e^{-\frac{|x|^2}{2}}$ in $u_{k+1}$. Indeed,
  if we write
  $$u_{k+1}=v(q)e^{-\frac{|x|^2}{2}}+w(x,q),$$ where for every $q\in
  \R^r$ one has $w(q,\cdot)\perp e^{-\frac{|\cdot|^2}{2}},$ the equation
  reduces to
  $$(A_0-\mu_0)w+\cdots+(A_{k+1}-\mu_{k+1})u_0=0.$$
  Freezing $q$ and taking the scalar product with $x\mapsto
  e^{-\frac{|x|^2}{2}}$ yields
  $$\lambda_{k+1}=C_{k+1}.$$

  Then, with $q$ still frozen one has $(A_0-\mu_0)w=r.h.s$ where the
  r.h.s is orthogonal to the ground state of $A_0$, which allow us to
  solve for $w$.

  If the r.h.s is $u_0$ times a polynomial in $(q,x)$, then the same
  holds for $w$ (in particular, for all $i$ one has $A_iw\in L^2$ so
  that it makes sense to proceed with the induction).

  It remains to choose $v$ so that $u_{k+1}$ satisfies the
  orthogonality constraint above, in order to be able to build the
  next terms.

  Since $\mu_1=0$ and $A_1u_0=0$, the terms $i=1$ vanish so that the
  first integral in which $u_{k+1}$ appears is not the next one but
  the one after it:
  $$\int_{\R^{n-r}}\overline{u_0}(x,q)\left(\sum_{i=2}^{k+3}\left[A_iu_{k+3-i}\right](q,x)-\sum_{i=2}^{k+2}\mu_iu_{k+3-i}(q,x)\right)\dd
  x.$$
  Hence, one wants to solve
  $$\int_{\R^{n-r}}e^{-\frac{|x|^2}{2}}[(A_2-\mu_2)ve^{-\frac{|\cdot|^2}{2}}](q,x)=F(q)+C_{k+3}e^{-\phi(q)},$$
  with \begin{multline*}F(q)=-\int_{\R^{n-r}}e^{-\frac{|x|^2}{2}}\left([(A_2-\mu_2)w](x,q)+\sum_{i=3}^{k+3}[A_iu_{k+3-i}](x,q)\right.\\-\left.\sum_{i=3}^{k+2}\mu_i
      u_{k+3-i}(x,q)\right)\dd x.\end{multline*}

    The symbol $a_2$ decomposes into a quadratic symbol in $(q,p)$,
    and an odd polynomial in $(x,\xi)$. The latter does not contribute
    to the integral in the left-hand-side, and the former commutes
    with multiplication by $e^{-\frac{|x|^2}{2}}$, so that
    $$\int_{\R^{n-r}}e^{-\frac{|x|^2}{2}}[(A_2-\mu_2)ve^{-\frac{|\cdot|^2}{2}}](q,x)=C_{n-r}\left(Q_S(iD)+\frac
      12\Hess(g_1+\frac
      12\sum_{i=1}^{n-r}\lambda_i)(q)-\mu_2\right)v.$$
    The equation on $v$ is then
    \[
      \left(Q_S(iD)+\frac 12\Hess(\mu\circ
        \rho)(0)(q)-\mu_2\right)v=C_{n-r}^{-1}\left(F(q)+C_{k+3}e^{-\phi(q)}\right).
        \]

    With $$C_{k+1}=-\langle e^{-\phi(q)},F(q)\rangle,$$ one
    has $$F-C_{k+1}e^{-\phi}\perp e^{-\phi},$$ so that one can solve
    for $v$.

    Again, if $u_0,\ldots,u_{k}$ and $w$ are $u_0$ times a polynomial
    function in $(x,q)$, then $F$ is $e^{-\phi}$ times a
    polynomial function, so that the same is true for $v$. This
    concludes the construction by induction.

    The estimation of the error terms stems directly from the fact
    that the terms $u_k$ are polynomials time a function with Gaussian
    decay. Hence, this formal construction yields approximate eigenfunctions.
  \end{proof}

 Before we show that the almost eigenfunction computed in Proposition
 \ref{prop:eigenf-reg} corresponds indeed to the lowest eigenvalue,
 let us use the quantum maps $\mathfrak{S}_N$ to obtain upper and lower
 bounds for $T_N(h)$, which will be useful in Section 7.

\begin{prop}\label{prop:Weylreg}
  For $t>0$ let $A^{reg}_N$ the following operator on $L^2(\R^r)$:

\[
A^{reg}_N=Op_W^{N^{-1}}\left(|p|^2+N^{-1}|q|^2\right)
\]

Under the conditions of Proposition \ref{prop:eigenf-reg}, there
exists $a_0>0$, and two constants $0<c<C$ such that, for any $N$,  for
any $a<a_0$, for
any normalized $u\in L^2(X)$ supported in
  $B(P_0,a)\times \S^1$, with $v=\mathcal{B}_N^{-1}\mathfrak{S}_N^{-1}u$, one has:
\begin{multline*}
c\langle
  v,A^{reg}_Nv\rangle
  -C\langle v,Op^{N^{-1}}_W(|N^{-\frac 12},x,\xi|^3)v\rangle
  +c\left(\langle v,Op_W^{N^{-1}}(|x|^2+|\xi|^2)v\rangle-N^{-1}\frac{n-r}{2}\right)
  \\\leq \langle u,hu\rangle-N^{-1}\mu(P_0)
\end{multline*}

In addition, the following bound holds:
\begin{multline*}
c\langle v,A^{reg}_Nv\rangle+\langle v,Op_W^{N^{-1}}(Q_F(0)(x,\xi))v\rangle -C\langle v,Op^{N^{-1}}_W(|N^{-\frac
  12},x,\xi|^3)v\rangle\\ -aC\left(\langle
  v,Op_W^{N^{-1}}(|x|^2+|\xi|^2)v\rangle -N^{-1}\frac{n-r}{2}\right)\\\leq \langle
u,hu\rangle-N^{-1}\mu(P_0)+\frac{N^{-1}}{2}\sum_i\lambda_i(0)\\
\leq C\langle v,A^{reg}_Nv\rangle+\langle v,Op_W^{N^{-1}}(Q_F(0)(x,\xi))v\rangle +C\langle v,Op^{N^{-1}}_W(|N^{-\frac
  12},x,\xi|^3)v\rangle\\ +aC\left(\langle
  v,Op_W^{N^{-1}}(|x|^2+|\xi|^2)v\rangle -N^{-1}\frac{n-r}{2}\right)
\end{multline*}

Here, $O(|x,\xi,N^{-\frac 12}|^3)$ stands for $O(|x,\xi|^3+N^{-\frac 32})$.
\end{prop}
\begin{proof}
  Let us prove the first lower bound. As
  \[
    g_0(q,p,x,\xi)=Q_F(q)(x,\xi)+Q_S(q)(p)+O(|p,x,\xi|^3),
  \]
  one has first, by a lower bound on $Op_W^{N^{-1}}(Q_F(q)(x,\xi))$,
  \begin{multline*}
    \langle v,Op_W^{N^{-1}}(g_0)v\rangle \\\geq c\langle
    v,Op_W^{N^{-1}}(|p|^2)v\rangle+\frac{N^{-1}}{2}\langle v,\sum_i
    \lambda_i(q),v\rangle+\langle v,Op_W^{N^{-1}}(|x,\xi,N^{-\frac 12}|^3)v\rangle.
    \end{multline*}
    Let us make this bound more precise. Since $\lambda_i(0)>0\, \forall
    i$, for $q$ small enough one has
\[Op_W^{N^{-1}}(Q_F(q)(x,\xi))-\frac{N^{-1}}{2}\sum_i\langle v,\lambda_i(q)v\rangle\geq cOp_W^{N^{-1}}(|x|^2+|\xi|^2)-\frac{N^{-1}}{2}c(n-r).\]

Hence
     \begin{multline*}
    \langle v,Op_W^{N^{-1}}(g_0)v\rangle \geq c\langle
    v,Op_W^{N^{-1}}(|p|^2)v\rangle+\frac{N^{-1}}{2}\langle v,\sum_i
    \lambda_i(q),v\rangle\\+c\langle
    v,Op_W^{N^{-1}}(|x|^2+|\xi|^2)\rangle-c\frac{N^{-1}}{2}(n-r)-C\langle
    v,Op_W^{N^{-1}}(|x,\xi,N^{-\frac 12}|^3)v\rangle.
  \end{multline*}

  Recall from Proposition \ref{prop:Melin-reg} that
  $g_1=\mu(\rho(q,0,0,0))-\frac 12 \sum_i\lambda_i(q)+O(|x,p,\xi|).$

  Hence,
  \begin{multline*}
  \langle v,Op_W^{N^{-1}}(g)v\rangle \geq c\langle
    v,Op_W^{N^{-1}}(|p|^2)v\rangle+\frac{N^{-1}}{2}\langle
    v,\mu(\rho(q,0,0,0)),v\rangle\\+c\langle v,Op_W^{N^{-1}}(|x|^2+|\xi|^2)-c\frac{N^{-1}}{2}(n-r)-C\langle
    v,Op_W^{N^{-1}}(|N^{-\frac 12},x,\xi|^3)v\rangle.
  \end{multline*}

  As $\mu(\rho(q,0,0,0))\geq \mu(P_0)+c|q|^2$, this yields the lower
  bound.

We now turn to the second estimate. This requires a bound on
\[\langle v,Op_W^{N^{-1}}(Q_F(q)(x,\xi)-Q_F(0)(x,\xi))v\rangle.\]
Since $Q_F$ has been diagonalised this operator is diagonal in the
Hilbert base given by the Hermite functions. Let us write
\[v=\sum_{\nu\in \N^{n-r}}\alpha_{\nu}(q)H_{\nu,N}(x),\]
Where $(H_{\nu,N})_{\nu}$ denote the Hilbert base of $\R^{n-r}$ given
by the Hermite functions. Then
\begin{multline*}
\langle v,Op_W^{N^{-1}}(Q_F(q)(x,\xi)-Q_F(0)(x,\xi))v\rangle\\
=\sum_{\nu}\langle \alpha_{\nu}(q),\nu\cdot
(\lambda(q)-\lambda(0))\alpha_{\nu}(q)\rangle+\frac 12 \langle v,\sum_{i}\lambda_i(q)-\lambda_i(0)v\rangle.
\end{multline*}
The key point is
\[\left|\sum_{\nu}\langle \alpha_{\nu}(q),\nu\cdot
  (\lambda(q)-\lambda(0))\alpha_{\nu}(q)\rangle\right| \leq
  Ca\sum_{\nu}\langle \alpha_{\nu}(q),|\nu| \alpha_{\nu}(q)\rangle.\]

The right-hand term is then equal to
\[Ca\left(\langle v,Op_W^{N^{-1}}(|x|^2+|\xi|^2)
    v\rangle-N^{-1}\frac{n-r}{2}\right).\]
This yields the desired control since
\[g_1=\mu(q)-\frac 12 \sum_i\lambda_{i}(q)+O(|p,x,\xi|).\]
\end{proof}


\subsection{Spectral gap}

It only remains to show that the sequence of almost eigenfunctions
given by Proposition \ref{prop:eigenf-reg} corresponds to the first
eigenvalue of $T_N(h)$.

\begin{prop}\label{prop:spgap-reg}
 Let $h\geq 0$ be such that the minimum of the Melin value $\mu$ is only reached at
  one point, which is a miniwell for $h$.
  
  Let $(\mu_i)$ be the real sequence constructed in the previous
  proposition, and let $\lambda_{\min}$ be the first eigenvalue of
  $T_N(h)$.

  Then $$\lambda_{\min}\sim N^{-1}\sum_{i=0}^{\infty}N^{-\frac
  i4}\mu_i.$$ Moreover, there exists $c>0$ such that,
for every $N$, one has $$\dist(\lambda_{\min},Sp(T_N(h))\setminus
\{\lambda_{\min}\})\geq c N^{-\frac 32}.$$
\end{prop}

\begin{preuve}
  Let us show that any function orthogonal to the one proposed in
  Proposition \ref{prop:eigenf-reg} has an energy which is larger by at least
  $cN^{-\frac 32}$.

Let $(v_N)$ be a sequence of unit vectors in $L^2(\R^n)$. If $$\langle
v_N,Op^{1}_W(\tilde{g}_N)v_N\rangle \leq N^{-1}\mu_0+CN^{-\frac
  32}$$ for some $C$, then $v_N=e^{-\frac{|x|^2}{2}}w_N(q)+O(N^{-\frac 12})$, with
$\|w_N\|_{L^2}=1+O(N^{-\frac 12})$.

If $C-\mu_2$ is strictly smaller than the spectral gap of the quadratic
operator $$Op_W^1\left(Q_S(0)(p)+\frac 12\Hess \mu\circ \rho(\cdot,0,0,0)(q)\right),$$ then
$\langle w_N,e^{-\phi(q)}\rangle\geq a$ for some $a>0$ independent of $N$,
which concludes the proof.
\end{preuve}


\section{A degenerate case}
In this section we treat a case in which the zero set of
the symbol is not a submanifold. The
local hypotheses on the symbol are as follows:

\begin{defn}\label{defn:simple-crossing}
  Let $h\in C^{\infty}(M,\R^+)$ and $P_0\in M$. The zero set of $h$ is
  said to have a \emph{simple crossing} at $P_0$ if there is an open
  set $U$ containing $P_0$ such that:
\begin{itemize}
\item $\{h=0\}\cap U=Z_1\cup Z_2$, where $Z_1$ and $Z_2$ are two
  pieces of smooth isotropic submanifolds of $M$.
\item $Z_1\cap Z_2=\{P_0\}$ and $T_{P_0}Z_1\cap T_{P_0}Z_2=\{0\}$.
\item $T_{P_0}Z_1\oplus T_{P_0}Z_2$ is isotropic.
\item For $i=1,2$, on all of $Z_i\setminus \{P_0\}$, $h$ vanishes at order exactly
  $2$ on $Z_i$.
\item There is $c>0$ such that, for all $x\in Z_1\cup Z_2$, one has:
\[
\mu(x)-\mu(P_0)\geq c \dist(P_0,x).
\]
\end{itemize}

\end{defn}
The last condition may seem very strong. However, $\mu$ is typically
only Lipschitz-continuous at the intersection. A typical example is
\[
h(q_1,q_2,p_1,p_2)=p_1^2+p_2^2+q_1^2q_2^2,
\]
where along
$\{q_1,0,0,0\}$ one has $\mu(q_1)=|q_1|+1.$
We exclude on purpose situations like
$\mu(q_1)=1+|q_1|-q_1+q_1^2$, which grows like $|q_1|$ for $q_1<0$ but
grows like $q_1^2$ for $q_1>0$.

Under the hypotheses of Definition \ref{defn:simple-crossing}, we first give a symplectic normal form of $h$
near $P_0$, then a description of the first eigenvector and eigenvalue
of $T_N(h)$.

\subsection{Symplectic normal form}
Let $Q\geq 0$ be a semidefinite positive quadratic form on $(\R^{2n},\omega)$, and
$(e_i,f_i)$ a symplectic basis of $\R^{2n}$ which diagonalises $Q$:
\[
Q\left(\sum_{i=1}^n
  q_ie_i+p_if_i\right)=\sum_{i=r+1}^{r'}p_i^2+\sum_{i=r'+1}^n\lambda_i(q_i^2+p_i^2),
\]
\[
\forall i, \lambda_i\neq 0.
\]
Let $M$ denote the matrix of $Q$ in the canonical basis. Then
\[
\{\pm i\lambda_{r'+1},\ldots,\pm i\lambda_{n}\}=
\sigma(JM)\setminus \{0\}.
\]
More precisely, if $E_{\lambda}$ denotes the (complex)
eigenspace of $JM$ with eigenvalue $\lambda$, then
\[
E_{i\lambda_j}\oplus
E_{-i\lambda_j}=\Span_{\C}((e_k,f_k),k>r',\lambda_k=\lambda_j).
\]
Moreover, Jordan blocks never occur for nonzero eigenvalues. Hence,
\begin{prop}
  If $Q:\R^m\mapsto S^+_{2n}(\R)$ is a smooth parameter-dependent
  semipositive quadratic form on $(\R^{2n},\omega)$, and if the $d$-th
  largest symplectic eigenvalue $\lambda_{n-d+1}$ (with multiplicity)
  never crosses the $d+1$-th largest symplectic eigenvalue, then the
  $2d$ last vectors $(e_{n-d+1},f_{n-d+1},\ldots,e_n,f_n)$ of a
  symplectic basis diagonalising $Q$ depend smoothly on the parameter,
  up to reordering of the $d$ largest symplectic eigenvalues.
\end{prop}

One can build a symplectic normal form as previously, under
the conditions above.

\begin{prop}\label{prop:cross-normal-form}
  Let $h$ satisfy the simple crossing conditions of Definition
  \ref{defn:simple-crossing}, and let
  \begin{align*}
    r_1&=\dim(Z_1)\\
    r_2&=\dim(Z_2).
  \end{align*}
  Then there is
  an open set $V \subset U$, containing $P_0$, and a symplectic map
\[
\sigma:V\mapsto\R^{2r_1}\times \R^{2r_2}\times \R^{2(n-r_1-r_2)}
\]
  such that
\begin{multline*}
h \circ
\sigma^{-1}(q_1,p_1,q_2,p_2,x,\xi)\\
=\sum_{i=1}^{n-r_1-r_2}\lambda_i({ q_1},{ q_2})(x_i^2+\xi_i^2)+
Q_S(q_1,q_2)({ p_1},{ p_2})\\
+\sum_{i,j=1}^{r_1}\sum_{k,l=1}^{r_2}\alpha_{ijkl}q_{1,i}q_{1,j}q_{2,k}q_{2,l}\\
+
O(\|{ x},{ \xi},{ p_1},{ p_2}\|^3)+O(\|{ q_1}\|^2\|{ q_2}\|^2(\|{ q_1}\|+\|{ q_2}\|)).
\end{multline*}
Moreover, for every $(q_1,q_2)\in (\R^{r_1}\setminus
\{0\})\times(\R^{r_2}\setminus \{0\})$ small enough, the matrices given by
$\left[\sum_{i,j}\alpha_{ijkl}q_{1,i}q_{1,j}\right]_{k,l}$ and
$\left[\sum_{k,l}\alpha_{ijkl}q_{2,k}q_{2,l}\right]_{i,j}$ are positive.
\end{prop}
\begin{proof}
  At $P_0$, there are exactly $n-r_1-r_2$ nonzero symplectic
  eigenvalues (with multiplicity) for the Hessian of $h$. Hence, in small
  neighbourhoods $V_1$ and $V_2$ of $P_0$ in $Z_1$ and $Z_2$, there is no crossing between the
  $r_1+r_2$-th largest eigenvalue and the one immediatly below. From
  the previous lemma there is a smooth choice of symplectic
  eigenvectors for the $r_1+r_2$ largest symplectic eigenvalues, which
  span a symplectic subbundle of $T_{Z_1\cup Z_2}M$. The symplectic
  orthogonal $\mathcal{F}$ of this bundle contains $TZ_1\cup TZ_2$ since the latter
  consists of zero vectors for the quadratic form. $\mathcal{F}$ forms a distribution, which is integrable along $Z_1\cup
  Z_2\setminus P_0$ as in Proposition \ref{prop:coords}. Let us
  show that $\mathcal{F}$ is also integrable at $P_0$.

  We take a first set of local coordinates $(q_1,q_2,...)$ such that
  $Z_1=\{q_1,0,0\}$ and $Z_2=\{0,q_2,0\}$. 

  Then, the restriction of the Hessian of $h$ at $(q_1,0,0)$ to
  $\{0,q_2,0\}$ is $O(|q_1|^2)$ as $q_1$ approaches zero. In
  particular, the distance in the Grassmannian from $\{0,q_2,0\}$ to
  the distribution above at $(q_1,0,0)$ is $O(|q_1|^2)$.
  Hence, the first differential of $\mathcal{F}$ at $P_0$
  in the directions contained in $Z_1$ and $Z_2$ are zero, so that the
  distribution is integrable at $P_0$.

  Hence, there is
  a symplectic manifold $S$, containing $Z_1\cup Z_2$, so that
  $T_{Z_1\cup Z_2}S$ is the symplectic orthogonal of the bundle above.

  As in the proof of Proposition \ref{prop:coords}, one first considers
  a symplectic map which
  \begin{itemize}
  \item sends $Z_1$ to $\R^{r_1}\times \{0,0,0,0,0\},$
  \item sends $Z_2$ to $\{0,0\}\times \R^{r_2}\times \{0,0,0\},$
\item sends $S$ to $\R^{2r_1}\times \R^{2r_2}\times \{0,0\}.$
\end{itemize}
Let us define
\[\hat{Z}=\R^{r_1}\times \{0\}\times \R^{r_2}\times\{0,0,0\}.\]

Then $\hat{Z}$ is isotropic. Along $\hat{Z}$ one wishes to find the
$n-r_1-r_2$ largest symplectic eigenpairs of the Hessian matrix of
$h$. However, this Hessian is not necessarily semi-positive definite
on $\hat{Z}$. To tackle this issue, we first formulate the span
$\mathcal{D}$ of what will be the largest symplectic eigenvectors by
a variatonal formulation.

We first observe that there is a symplectic subspace of
$\R^{2n}$, of dimension $2(n-r_1-r_2)$, on which the restriction of
$\Hess(h)(0)$ is definite positive.

By continuity, for $z\in \hat{Z}$ small enough, the set of
subspaces of $\R^{2n}$ of dimension $2(n-r_1-r_2)$, on which the
restriction of $\Hess(h)(z)$ is semi-definite positive, is non-empty. Let us call
$G^+(z)$ this closed subset of the Grassmannian. For each element
$F\in G^+(z)$, we pick a basis of $F$ under which we denote by $M_{F}$
the matrix elements of $\Hess(h)(z)$ and $J_F$ the matrix elements of
the symplectic form $\omega$. Since $M_F$ has a square root among
semidefinite positive matrices, the matrix $J_FM_F$ has purely
imaginary spectrum (since it is conjugated with the antisymmetric
matrix $M_F^{\frac
  12}J_FM_F^{\frac 12}$). We then let $\kappa(F)=\dist(0,Sp(J_FM_F))$.

We claim that maximising $\kappa(F)$ on $G^+(z)$ leads to
$\mathcal{D}(z)$.
We first note that $\kappa(F)=0$ on the boundary of $G^+(z)$ (on which
the matrix $M_F$ is singular), and moreover $\kappa(F)=0$ if $F$
is not a symplectic subspace of $R^{2n}$.

For $z=0$, the function $\kappa$ has only one maximum which is
positive and non-degenerate. Hence the same holds for $z$ small since
$\kappa$ at $z$ is $C^2$-close to $\kappa$ at zero. The unique
maximal point $\mathcal{D}(z)$ depends smoothly on $z$ and is
transverse to $T_z\hat{Z}$. From the expression of the differential of
$\kappa$, the symplectic orthogonal of
$\mathcal{D}(z)$ is also its $\Hess(h)(z)$-orthogonal.

The Hessian matrix of $h$ at $z$, when restricted on $\mathcal{D}(z)$, is definite
positive, hence has symplectic eigenpairs which depend smoothly on $z$.
Hence, as in Proposition \ref{prop:coords} one can find a symplectic map
which is identity on $S$ and which diagonalises the fast modes along $\hat{Z}$.

Then it only remains to study the behaviour of $h\circ \sigma^{-1}$
on $\hat{Z}$, near $0$.
As $h$ is non-negative and vanishes exactly on $Z_1\cup Z_2$, one has
\[
h\circ \sigma^{-1}(q_1,0,q_2,0,0,0)=O(\|q_1\|^2\|q_2\|^2).
\]
The
dominant term is then of the form:
\[
\sum_{i,j,k,l}\alpha_{ijkl}q_{1,i}q_{1,j}q_{2,k}q_{2,l}.
\]

The positivity conditions on the tensor $\alpha$ are then directly given by the fact that $h$
vanishes at order $2$ on $Z_1\setminus \{P_0\}$ and $Z_2\setminus \{P_0\}$.
\end{proof}

  One can easily adapt Definition \ref{defn:simple-crossing} to the
   case of a crossing along a submanifold.

 \begin{defn}[Crossing along a submanifold]\label{defn:cross-subm}
     Let $h\in C^{\infty}(M,\R^+)$ and $P_0\in M$. The zero set of $h$ is
   said to \emph{cross along a submanifold} near $P_0$ if there is an open
   set $U$ containing $P_0$ such that:
 \begin{itemize}
 \item $\{h=0\}\cap U=Z_1\cup Z_2$, where $Z_1$ and $Z_2$ are two
   pieces of smooth isotropic submanifolds of $M$.
 \item $Z_1\cap Z_2=Z_3$ is a piece of smooth submanifold containing
   $P_0$. For each $x\in Z_3$, one has $T_xZ_3=T_xZ_1\cap T_xZ_2$.
 \item For each $x\in Z_3$, the space $T_{x}Z_1+ T_{x}Z_2$ is isotropic.
 \item For $i=1,2$, on all of $Z_i\setminus Z_3$, $h$ vanishes at order exactly
   $2$ on $Z_i$.
 \item There is $c>0$ such that, for all $x\in Z_1\cup Z_2$, one has:
 \[
 \mu(x)-\mu(P_0)\geq c \dist(Z_3,x).
 \]
 \end{itemize}
 \end{defn}
 With this definition one can find a normal form as previously:

   \begin{prop}
   Let $h$ satisfy the conditions of Definition \ref{defn:cross-subm},
   and let
   \begin{align*}
     r_1&=\dim(Z_1)-\dim(Z_3)\\
     r_2&=\dim(Z_2)-\dim(Z_3)\\
     r_3&=\dim(Z_3).
   \end{align*}
   Then there is
   an open set $V \subset U$, containing $P_0$, and a symplectic map
 \[
 \sigma:V\mapsto\R^{2r_1}\times \R^{2r_2}\times \R^{2r_3}\times\R^{2(n-r_1-r_2-r_3)}
 \]
   such that
 \begin{multline*}
 h \circ
 \sigma^{-1}(q_1,p_1,q_2,p_2,q_3,p_3,x,\xi)\\
 =\sum_{i=1}^{n-r_1-r_2}\lambda_i({ q_1},{ q_2},q_3)(x_i^2+\xi_i^2)+
 Q_a(q_1,q_2,q_3)({ p_1},{ p_2})+Q_b(q_1,q_2,q_3)(p_3)\\
 +\sum_{i,j=1}^{r_1}\sum_{k,l=1}^{r_2}\alpha_{ijkl}(q_3)q_{1,i}q_{1,j}q_{2,k}q_{2,l}\\
 +
 O(\|{ x},{ \xi},{ p_1},{ p_2}\|^3)+O(\|{ q_1}\|^2\|{ q_2}\|^2\cdot\|( q_1, q_2)\|).
 \end{multline*}
 Moreover, for every $q_3\in \R^{r_3}$ small enough, for every $(q_1,q_2)\in (\R^{r_1}\setminus
 \{0\})\times(\R^{r_2}\setminus \{0\})$ small enough, the matrices given by
 $\left[\sum_{i,j}\alpha_{ijkl}q_{1,i}q_{1,j}\right]_{k,l}$ and
 $\left[\sum_{k,l}\alpha_{ijkl}q_{2,k}q_{2,l}\right]_{i,j}$ are positive.
 \end{prop}
 \begin{proof}
   One can repeat the proof of Proposition \ref{prop:cross-normal-form}
   with $\hat{Z}$ containing $Z_1\cup Z_2$. This yields the desired normal
   form, except for $Q_a$ and $Q_b$ which are replaced with a more
   general quadratic function $Q_S(q_1,q_2,q_3)(p_1,p_2,p_3)$.

   In order to separate $p_3$ from $(p_1,p_2)$, we first apply Lemma
   \ref{prop:quadred} with $Z_3$ as parameter space, in order to find, for every $q_3$, a
   decomposition of $\R^{2(r_1+r_2+r_3)}$ into a sum of symplectic
   spaces $S\oplus S'$ with $TZ_3\subset S$, and so that $S$ and $S'$
   are orthogonal for $Q_S(q_1=0,q_2=0,q_3).$
   There is a symplectomorphism $\sigma$ sending $S$ into $\R^{2r_1}$
   such that $\sigma|_{Z_3}=Id$.

   The map $\sigma$ may distort $\{h=0\}$, which can be flattened
   again on each space $\R^{2r_1}\times \R^{2r_2}\times
   \{q_3,0\}\times \R^{2(n-r_1-r_2-r_3)}$, with smooth dependence on
   $q_3$.

   Again, this leads to a symplectomorphism on a neighbourhood of zero
   in $\R^{2n}$, up to a small correction controlled by $O(|p_3|^2)$.
 \end{proof}
 \begin{rem}[More general degenerate crossings]
  Simple crossings (and crossings along submanifolds) are not stable
  by Cartesian products, which leads to a slightly more general
  situation (see Remark \ref{rem:crossprod}).

  On the other hand, one could try to deal with symbols whose zero set
  form a \emph{stratified manifold}, which are defined recursively: a
  stratified manifold is a union of smooth manifolds with clean
  intersections, such that the union of all intersections is itself a
  stratified manifold. The boundary of a hypercube is an instance of a
  stratified manifold.

  In this respect, a model case for a stratified situation of degree
  three is
  \[
    p_1^2+p_2^2+p_3^2 + q_1^2q_2^2q_3^2,
  \]
  with zero set $\{p_1=0,p_2=0,p_3=0,q_i=0\}$ for every $i=1,2,3$.
  
  For this operator, the ground state is rapidly decreasing at
  infinity \cite{helffer_decroissance_1992} but this is not due to
  subprincipal effects. Indeed, in this setting, $\mu$ is constant
  along the three axes. If we add a generic transverse quadratic
  operator $Q_q(x,\xi)$, the subprincipal effect will dominate and
  has no reason to select the point $\{q=0\}$, as opposed to the
  simple crossing case where an open set of symbols sharing the same
  minimal set have minimal Melin value at the crossing point.
 \end{rem}

\subsection{Study of the model operator}
\label{subsec:study-model-operator}
As Proposition \ref{prop:cross-normal-form} suggests, the following
operators play an important role in the study of the crossing case:
\[
P=Q(iD)+\sum_{i,j=1}^{r_1}\sum_{k,l=1}^{r_2}\alpha_{ijkl}q_{1,i}q_{1,j}q_{2,k}q_{2,l}+\sum_{i=1}^{r_1}L_{1,i}
q_{1,i}+\sum_{i=1}^{r_2}L_{2,i}q_{2,i},
\]
acting on
$L^2(\R^{r_1+r_2})$, where $D$ is the differentiation operator and
$Q>0$ is a quadratic form.
The linear form $L$ will appear as an effect of the subprincipal
symbol, as we will see later.

Let $Q_1$ and $Q_2$ denote the restrictions of the quadratic form $Q$
on $\R^{r_1}\times \{0\}$ and $\{0\}\times \R^{r_2}$, respectively.
Throughout this Subsection we impose the following conditions on $P$:
\begin{itemize}
\item For every $(q_1,q_2)$, one has
\[
\sum_{i,j=1}^{r_1}\sum_{i,j=1}^{r_2}\alpha_{ijkl}q_{1,i}q_{1,j}q_{2,k}q_{2,l}\geq 0.
\]
\item For every $q_1\neq 0$, one has
\[
Q_2(iD)+\sum_{ijkl}q_{1,i}q_{1,j}q_{2,k}q_{2,l}>
  \sum_{i=1}^{r_1}L_{1,i}q_{1,i},
\]
\item For every $q_2\neq 0$, one has
\[
Q_1(iD)+\sum_{ijkl}q_{1,i}q_{1,j}q_{2,k}q_{2,l}>
  \sum_{i=1}^{r_1}L_{2,i}q_{2,i},
\]
\end{itemize}

\begin{rem}\label{rem:crossprod}
  These conditions are weaker than what Definition
\ref{defn:simple-crossing} calls for. There does not need to be a simple
crossing in this case as the following example illustrates:
\[
P=-\Delta+q_{1,1}^2q_{2,1}^2+q_{1,2}^2q_{2,2}^2.
\]

There, the zero set of the symbol is a union of four isotropic
surfaces in $\R^8$, i.e. $\{p=0, q_{1,i}=0,q_{2,j}=0\}$ for all
$(i,j)\in \{1,2\}^2$.
\end{rem}
\begin{prop}
  Under the previous conditions, there exists $c>0$ such that
\[
P\geq c(Q(iD) + |q|).
\]
\end{prop}
\begin{proof}
  Let $Q_2$ be the restriction of the quadratic form $Q$ to $\{0\}\times
  \R^{r_2}$. One has $Q\geq Q_2$, hence $Q(iD)\geq Q_2(iD)$.
  By hypothesis,
\[
Q_2(iD)+\sum_{ijkl}q_{1,i}q_{1,j}q_{2,k}q_{2,l}>
  \sum_{i=1}^{r_1}L_{1,i}q_{1,i},
\]
and the infimum of the spectrum of
  the left hand side is $1$-homogeneous in $q_1$, so that
\[
Q_2(iD)+\sum_{ijkl}q_{1,i}q_{1,j}q_{2,k}q_{2,l}\geq (1-c)
  \sum_{i=1}^{r_1}L_{1,i}q_{1,i}+2c|q_1|
\]
for some $c>0$.
  In particular,
\[
P \geq c Q(iD)+2c|q_1|.
\]
  The same reasoning applies to $Q_2$, hence 
\[
2P \geq 2c Q(iD)+2c|q_1|+2c|q_2|,
\]
which allows us to conclude.
\end{proof}
One deduces immediately:
\begin{prop}
  The operator $P$ has compact resolvent. Its first eigenvalue is positive.
\end{prop}

We are now able to use Agmon estimates. In the particular case where
$Q$ is diagonal, the following result is contained in the
Helffer-Nourrigat theory \cite{helffer_decroissance_1992}, see also the
related results in \cite{morame_accuracy_2006}.
\begin{prop}\label{prop:Agmon}
 Let $\lambda_0$ be the first eigenvalue of $P$. There exists $c>0$ such that, if $u\in
  L^2(\R^{r_1+r_2})$, and $(C_{\beta})_{\beta \in \N^{r_1+r_2}}$ are such that $|\partial^{\beta}u(q)|\leq
  C_{\beta}e^{-c|q|^{3/2}}$ for all $q\in \R^{r_1+r_2},\beta\in \N^{r_1+r_2}$, then for any $f\in
  L^2(\R^{r_1+r_2})$ such that $(P-\lambda_0)f=u$, there exists
  $(C'_{\beta})_{\beta\in \N^{r_1+r_2}}>0$
  such that $|\partial^{\beta}f(q)|\leq C'e^{-c|q|^{3/2}}$ for every
  $q\in \R^{r_1+r_2},\beta\in \N^{r_1+r_2}.$
\end{prop}
\begin{proof}
With $\phi(q)=c|q|^{3/2}$, one has $Q(\vec{\nabla}\phi)\leq
c'|q|$. Hence $P-\lambda_0-Q(\vec{\nabla}\phi)$ is positive far from
zero, and one can use Agmon estimates as developed in \cite{agmon_lectures_2014}.
\end{proof}
We will also need the following two facts. Proposition
\ref{prop:simple} is an essential ingredient of Subsection
\ref{subsec:cross-first-eigenvalue} and Proposition \ref{prop:Weyl} is
necessary to compare the Weyl asymptotics with the regular case.
\begin{prop}\label{prop:simple}
  The first eigenvalue $\lambda_0$ of $P$ is simple.
\end{prop}
\begin{proof}
  This follows from an argument which is standard in the case
  $Q=Id$. Let $u_0\in L^2(\R^{r_1+r_2})$ be
  such that $Pu_0=\lambda_0 u_0$. Then $u_0$ is a minimizer of the
  Courant-Hilbert problem
\[
\min_{\|u\|_{L^2}=1,\,u\in H^1}\int
  Q(\vec{\nabla}u)+V|u|^2.
\]

  The set $\{u_0=0\}$ has zero Lebesgue measure from a standard Unique
  Continuation argument. The function $|u_0|$ is then also a minimizer of this quantity,
  since $\vec{\nabla}|u_0|=\pm \vec{\nabla}u_0$ whenever $u_0\neq 0$.

  Then $|u_0|$ itself belongs to the eigenspace of $P$ with value
  $\lambda_0$, which is (a priori) a finite-dimensional space of real
  analytic (complex-valued) functions. Hence, $|u_0|$ is real analytic
  so that $u_0=|u_0|e^{i\theta_0}$, with $\theta_0$ real analytic.

  Now
\[
\int \overline{u_0}Pu_0=\int |u_0|(P-Q(\nabla
  \theta_0))|u_0|=\lambda_0-\int Q(\nabla \theta_0)|u_0|^2.
\]

  As $\{|u_0|=0\}$ has zero Lebesgue measure and $Q>0$, the function
  $\theta_0$ is constant, so that $u_0$ and $|u_0|$ are colinear.

  To conclude, if $u_0$ and $u_1$ are two orthogonal eigenfunctions of
  $P$ with eigenvalue $\lambda_0$, then $|u_0|$ and $|u_1|$ are
  orthogonal with each other, and both have $\R^{r_1+r_2}$ as
support, so that either $u_0=0$ or $u_1=0$.
\end{proof}

\begin{prop}\label{prop:Weyl}
  Suppose $P$ satisfies the following two supplementary conditions:
  \begin{itemize}
  \item $r_1=r_2$.
  \item For every $(q_1,q_2)\in (\R^{r_1}\setminus
\{0\})\times(\R^{r_2}\setminus \{0\})$, the matrices given by\\
$\left[\sum_{i,j}\alpha_{ijkl}q_{1,i}q_{1,j}\right]_{k,l}$ and
$\left[\sum_{k,l}\alpha_{ijkl}q_{2,k}q_{2,l}\right]_{i,j}$ are positive.
  \end{itemize}

  Let $\Lambda>0$ and let $N_{\Lambda}$ denote the number of eigenvalues of $P$
  less than $\Lambda$ (with multiplicity).

  Then there are $C>c>0$ such that, as $\Lambda\to +\infty$, one has
\[
c\Lambda^{\frac 32 r_1}\log(\Lambda)\leq N_{\Lambda}\leq C\Lambda^{\frac 32 r_1}\log(\Lambda).
\]
\end{prop}
\begin{proof}
  Under the second supplementary condition, the quartic part of the
  potential is greater than $c|q_1|^2|q_2|^2$ for some $c>0$. 
  Hence, for some $C>0$ one has $N_{\lambda}\geq \tilde{N}_{\lambda}$, where
  $\tilde{N}_{\Lambda}$ counts the eigenvalues less than $\Lambda$ of
\[
-\Delta + |q_1|^2|q_2|^2+|q_1|+|q_2|.
\]

  On the other hand one clearly has $P\leq C(-\Delta +
  |q_1|^2|q_2|^2+|q_1|+|q_2|)$ for some $C>0$.

  Thus, the problem boils down to Weyl asymptotics for the elliptic
  operator $-\Delta+|q_1|^2|q_2|^2+|q_1|+|q_2|$. It suffices to control the volume of the sub-levels of
  its symbol:
\[
\{(q_1,q_2,p_1,p_2)\in \R^{4r_1},\,|p_1|^2+|p_2|^2 +
|q_1|^2|q_2|^2+|q_1|+|q_2|\leq \Lambda\}.
\]
We first study
\[
A_{\Lambda}=\{(q_1,q_2)\in \R^{r_1},\,
|q_1|^2|q_2|^2+|q_1|+|q_2|\leq \Lambda\}.
\]
Then, decomposing $A_{\Lambda}$ into $A_{\Lambda} \cap
  B(0,\Lambda^{\frac{r_1}{2}})$ and its complement set yields
\begin{align*}
Vol(A_{\Lambda})&\leq C\Lambda^{\frac {r_1}2}+2\int_{|q_1|\geq
                  \Lambda^{1/4}}Vol\{q_2,|q_1|^2|q_2|^2+|q_1|+|q_2|\leq
                  \Lambda\}.\\
&\leq C \Lambda^{\frac {r_1}2}+2C\int_{|q_1|\geq
  \Lambda^{1/4}}\left(\cfrac{\sqrt{\Lambda-|q_1|+|q_1|^{-1}}}{|q_1|}\right)^{r_1}\\
&\leq C
  \Lambda^{\frac{r_1}2}+2C\Lambda^{\frac{r_1}{2}}\int_{\Lambda^{-\frac
  34}}^2\frac 1x \dd x\\
&\leq C \Lambda^{\frac {r_1}2}\log(\Lambda).
\end{align*}
On the other hand,
\begin{align*}
  Vol(A_{\Lambda}) &\geq 2\int_{|q_1|\geq
                  \Lambda^{1/4}}Vol\{q_2,|q_1|^2|q_2|^2+|q_1|+|q_2|\leq
                  \Lambda\}.\\
&\geq 2c\int_{|q_1|\geq
  \Lambda^{1/4}}\left(\cfrac{\sqrt{\Lambda-|q_1|+|q_1|^{-1}}}{|q_1|}\right)^{r_1}\\
&\geq 2c\Lambda^{\frac{r_1}{2}}\int_{\Lambda^{-\frac
  34}}^2\frac 1x \dd x\\
&\geq c \Lambda^{\frac {r_1}2}\log(\Lambda).
\end{align*}
Integrating yields
\begin{multline*}Vol(\{(q_1,q_2,p_1,p_2)\in \R^{4r_1},\,|p_1|^2+|p_2|^2 +
|q_1|^2|q_2|^2+|q_1|+|q_2|\leq \Lambda\})\\\in [c\Lambda^{\frac
  32 r_1}\log(\Lambda), C \Lambda^{\frac
  32 r_1}\log(\Lambda)],\end{multline*}
hence the claim.
\end{proof}

\subsection{Approximate first
  eigenfunction}\label{subsec:cross-first-eigenvalue}

In this Subsection we give an expansion for the first eigenfunction
and eigenvalue in a crossing case, following the same strategy as
Subsection \ref{subsec:reg-first-eigenvalue}. We quantize the
symplectic map of Proposition \ref{prop:cross-normal-form} and we use
the Bargmann transform to reformulate the problem in the
pseudodifferential algebra, in which we squeeze the operator. This
time, the squeezing is of order $N^{\frac 16}$ along $(q_1,q_2)$,
with a concentration speed of $N^{-\frac 13+\epsilon}$ along the zero set,
instead of $N^{-\frac 14+\epsilon}$ as was seen in the regular case.
We then apply a perturbative argument to obtain the full expansion of the
first eigenvalue and eigenvector.

\begin{defn}
  For any choice $\mathfrak{S}_N$ of quantization of the map $\sigma$ of
  Proposition \ref{prop:cross-normal-form}, the classical symbol
  $g_{\mathfrak{S}}\sim \sum N^{-i}g_i$ on a neighbourhood $U$ of $0$ in
  $\R^{2n}$ is defined as follows: for any sequence $(u_N)_{N\geq 1}$ with microsupport in
  a compact set of $U$, the following holds:
\[
\mathcal{B}_N^{-1}\mathfrak{S}_N^{-1}T_N(h)\mathfrak{S}_N\mathcal{B}_Nu_N=Op_W^{N^{-1}}(g_{\mathfrak{S}})u_N
+ O(N^{-\infty}).
\]
\end{defn}
In what follows, we choose an arbitrary quantum map $\mathfrak{S}_N$,
and we write $g$ instead of $g_{\mathfrak{S}}.$

The subprincipal part $g_1$ is prescribed on $Z_1\cup Z_2$ by the
local Melin estimates.
\begin{prop}\label{prop:Melin-cross}
  Along $\sigma(Z_1)$, for $q_1$ close to zero, one has
\[
g_1(q_1,0,0,0,0,0)=\frac 12 \sum_i\lambda_i(q_1,0)+\frac 14 \tr(Q_S(q_1,0)).
\]
Along $\sigma(Z_2)$, for $q_2$ close to zero, one has
\[
g_1(0,0,q_2,0,0,0)=\frac 12
\sum_i\lambda_i(0,q_2)+\frac 14 \tr(Q_S(0,q_2)).
\]
\end{prop}
The proof is exactly the same as for Proposition \ref{prop:Melin-reg}.

Let us define
\begin{multline*}P=Q_S(0)(-iD_{q_1},-iD_{q_2})+\sum_{ijkl}\alpha_{ijkl}q_{1,i}q_{1,j}q_{2,k}q_{2,l}\\+\nabla\left(\frac
    12\sum_{i=1}^{n-r_1-r_2}\lambda_i+\frac
    14 \tr
  Q_S\right)_{q_1=q_2=0}\cdot (q_1,q_2).\end{multline*}

Then $P$ satisfies the hypotheses of Subsection
\ref{subsec:study-model-operator} and Proposition \ref{prop:Weyl}.
\begin{prop}\label{prop:eigenf-cross}
  Under the conditions of Definition \ref{defn:simple-crossing}, there
  exists $c>0$, a sequence $(u_i)\in (\C[X_1,\ldots,X_{n-r_1-r_2},L^2(\R^{r_1+r_2})])^{\N}$ and a real
  sequence $C_{i,\alpha,\beta}$ with
\[
\forall (i,\alpha,\beta,q)\in
  \N\times\N^{n-r_1-r_2}\times \N^{r_1+r_2}\times
  \R^{r_1+r-2},\,|\partial^{\beta}u_{i,\alpha}(q)|\leq C_{i,\alpha,\beta}e^{-c|q|^{3/2}},
\]
and a
  sequence $(\mu_i)\in \R^{\N}$ with $\mu_0=\mu(P_0)$, $\mu_1=0$ and
  $\mu_2=\min \Sp(P)$, so that
\[
N^{\frac n2 - \frac{r_1+r_2}{6}}e^{-N|x|^2/2}\sum_{i=0}^{+\infty}N^{-\frac i6}u_i(N^{\frac
    12}x,N^{\frac 13}q)
\]
is an $O(N^{-\infty})$-eigenfunction of
  $Op_W^{N^{-1}}(g)$, with eigenvalue
\[
N^{-1}\sum_{i=0}^{+\infty}N^{-\frac i6}\mu_i.
\]
\end{prop}
This proposition provides an almost eigenfunction which we will show
to be associated to the lowest eigenvalue (see Proposition \ref{prop:order-cross}
). It is the main argument in the proof of Theorem
C; the concentration speed of this eigenfunction on zero, which is $N^{-\frac
  13}$, is the concentration speed of the lowest eigenvector of
$T_N(h)$ on the miniwell, because of Proposition \ref{lem:qmaps-speed}.

\begin{proof}
  As announced, let us squeeze $g$ by computing
\[
\tilde{g}=g(N^{-\frac
  13}q_1,N^{-\frac 23}p_1,N^{-\frac 13}q_2,N^{-\frac
    23}p_2,N^{-\frac 12}x,N^{-\frac 12}\xi).
\]
Grouping terms in the Taylor expansion yields, for any fixed $K\in
\N$,
\[
Op_W^{1}(\tilde{g})=N^{-1}\sum_{i=0}^{K} N^{-\frac
  i6}Op_W^1(a_i)+O(N^{-\frac {K+7}6}).
\]
The first terms are:
\begin{align*}a_0&=\sum_{i=1}^{n-r}\lambda_i(0)\left(x_i^2+\xi_i^2+\frac
                   12\right)+\frac 14 \tr(Q_S(0))\\
a_1&=0\\
a_2&=\sigma(P)\\
a_3&=R_3(x,\xi)+L(x,\xi).
\end{align*}
Here $P$ is as above, $R_3$ is a degree three polynomial and $L$ is a
linear form.

With $A_i=Op_W^1(a_i)$ let us solve by induction on $k$ the following
equation, where $(u_k)_{k\in \N}$ is as in the claim:
\[
\left(\sum N^{-\frac i6}(A_i-\mu_i)\right)\left(\sum N^{-\frac
    i6}u_i\right)=0.
\]

If $v_0$ is the (unique) ground state of $P$ then our starting point
is
\[
u_0=e^{-\frac {|x|^2}{2}}v_0,\,\mu_0=\min \Sp A_0,
\]
\[
u_1=0,\,\mu_1=0.
\]
Indeed $u_0$ is an
almost eigenvector for $Op_W^1(\tilde{g})$, with eigenvalue
$N^{-1}\mu_0+O(N^{-\frac 43}).$

Let us start an induction at $k=1$.
Suppose we have constructed the first $k$ terms of the expansion
$u_0,\ldots,u_k$ and $\mu_0,\ldots,\mu_k$, 
with $u_i\perp u_0$ for every $i$,
and suppose that, for some $C_k\in \R$, one has, for every $q\in
\R^{r_1+r_2}$,
\[
\int_{\R^{n-r_1-r_2}}\overline{u_0}(q,x)\left(\sum_{i=2}^{k+1}[A_iu_{k+1-i}](q,x)-\sum_{i=2}^{k}[\mu_iu_{k+1-i}](q,x)\right)
\dd x=C_{k+1}|v_0(q)|^2.
\]

Then the eigenvalue problem yields $u_{k+1}$ up to a function of the
form $v(q)e^{-\frac{|x|^2}{2}}.$
Indeed, writing $u_{k+1}(q,x)=v(q)e^{-\frac{|x|^2}{2}}+w(q,x)$, where
for every $q$ one has $w(q,\cdot)\perp e^{-\frac{|\cdot|^2}{2}}$, the
eigenvalue equation is
\[
(A_0-\mu_0)u_{k+1}+(A_2-\mu_2)u_{k-1}+\ldots+(A_{k+1}-\mu_{k+1})u_0=0
\]
for $u_{k+1}$ and $\mu_{k+1}$. First
$(A_0-\mu_0)v(q)e^{-\frac{|x|^2}{2}}=0$ so that
\[
(A_0-\mu_0)w+(A_2-\mu_2)u_{k-1}+\ldots+(A_{k+1}-\mu_{k+1})u_0=0
\]

By hypothesis, freezing the $q$ variable and taking the scalar product of this equation with $x\mapsto
e^{-\frac{|x|^2}{2}}$ yields $(C_{k+1}-\mu_{k+1})|v_0(q)|^2=0.$
Let $\mu_{k+1}=C_{k+1}$. Then, for every $q\in \R^{r_1+r_2}$, the function
\[
f_{k+1}:x\mapsto \sum_{i=2}^{k+1}[(A_i-\mu_i)u_{k+1-i}](q,x)
\]
is
orthogonal to $x\mapsto e^{-\frac {|x|^2}2}$. Hence
$w=(A_0-\mu_0)^{-1}f_{k+1}$ is well-defined and satisfies the
eigenvalue equation.

Moreover, from Proposition \ref{prop:Agmon}, if by induction $f_{k+1}$ is $e^{-\frac{|x|^2}{2}}$ times a
polynomial in $x$, and if any derivative of any coefficient decays as fast as
$e^{-c|q|^{3/2}}$, then the same is true for $w$.

At this point we need to check that, after the first step $k=1$, the value
$\mu_2$ is indeed $\min \Sp(P)$.

If $k=1$ then we are interested in the integral
\[
\int_{\R^{n-r_1-r_2}}e^{-\frac{|x|^2}2}v_0(q)[A_2u_0](q,x)\dd x =
  \min \Sp(P)|v_0(q)|^2,
\]
since $v_0$ is a ground state of $P$. This
  is indeed a constant function times $|v_0(q)|^2$, so that the
  induction hypothesis is satisfied at the first step, and $\mu_2=\min
  \Sp(P)$ as required.

Now recall $u_k(q,x)=v(q)e^{-\frac{|x|^2}{2}}+ w(q,x)$. The eigenvalue equation
in itself does not state any condition on $v$; however, to compute the
second next order, one needs to satisfy an orthogonality condition, i.e.

\[
\int_{\R^{n-r_1-r_2}}\overline{u_0}(q,x)\left(\sum_{i=2}^{k+3}[A_iu_{k+3-i}](q,x)-\sum_{i=2}^{k+2}[\mu_iu_{k+3-i}](q,x)\right)
\dd x=C_{k+3}|v_0(q)|^2.
\]
This is equivalent to
\[
\int_{\R^{n-r_1-r_2}}e^{-\frac{|x|^2}{2}}\left[(A_2-\mu_2)ve^{-{\frac{|x|^2}{2}}}\right](x,q)\dd
x=F(q)+C_{k+3}v_0(q).
\]

Now $a_2$ has no terms in $x$ or $\xi$ so the equation reduces to
\[
(A_2-\mu_2)v=F(q)+C_{k+3}v_0(q).
\]

Here,
\[
F(q)=\int_{\R^{n-r_1-r_2}}e^{-\frac{|x|^2}{2}}\left(\sum_{i=3}^{k+3}[A_iu_{k+3-i}](q,x)-\sum_{i=3}^{k+2}[\mu_iu_{k+3-i}](q,x)\right)
\dd x,
\]
so that $|\partial_{\beta}F(q)|\leq C_{\beta} e^{-c|q|^{3/2}}$.

To solve this equation, one takes $C_{k+3}=-\langle v_0,F\rangle$,
then the r.h.s is orthogonal to $v_0$, so that one can solve for $v$
(indeed, $\mu_2$ is a simple eigenvalue of $A_2$ by Proposition \ref{prop:simple}).

Then, by Proposition \ref{prop:Agmon}, one has, for all $\beta\in
\N^{r_1+r_2}$, for some $C_{\beta}$, that
$|\partial^{\beta}v(q)|\leq C_{\beta} e^{-c|q|^{3/2}}$ for all $q\in \R^{r_1+r_2}$.
This ends the induction.

The previous considerations were formal, but the decay properties of
the functions $u_k$ imply that
$A_ju_k\in L^2$ for every $j$ and $k$, which concludes the proof.
\end{proof}

\begin{prop}\label{prop:Weylcross}
  For $t>0$ let $A^{cross}_N$ the following operator on $L^2(\R^{r_1+r_2})$:

\[
A^{cross}_N=Op_W^{N^{-1}}\left(|p|^2+|q_1|^2|q_2|^2\right)
\]

Under the conditions of Definition \ref{defn:simple-crossing} and
Proposition \ref{prop:Weyl}, there
exists $a_0>0$, and two constants $0<c<C$ such that, for any $N$, for
any $a<a_0$, for
any normalized $u\in L^2(X)$ supported in
  $B(P_0,a)\times \S^1$, with $v=\mathcal{B}_N^{-1}\mathfrak{S}_N^{-1}u$, one has:
\begin{multline*}
c\langle
v,A^{cross}_Nv\rangle
+c\left(\langle v,Op_W^{N^{-1}}(|x|^2+|\xi|^2)v\rangle
  -N^{-1}\frac{n-r}{2}\right)\\  -C\langle v,Op^{N^{-1}}_W(|N^{-\frac 12},x,\xi|^3)v\rangle-CN^{\frac 43}
  \\\leq \langle u,hu\rangle-N^{-1}\mu(P_0)
\end{multline*}

In addition, the following bound holds:
\begin{multline*}
c\langle v,A^{cross}_Nv\rangle+\langle v,Op_W^{N^{-1}}(Q_F(0)(x,\xi))v\rangle -C\langle v,Op^{N^{-1}}_W(|N^{-\frac
  12},x,\xi|^3)v\rangle\\ -aC\left(\langle
  v,Op_W^{N^{-1}}(|x|^2+|\xi|^2)v\rangle
  -N^{-1}\frac{n-r}{2}\right)-CN^{-\frac 43}\\\leq \langle
u,hu\rangle-N^{-1}\mu(P_0)+\frac{N^{-1}}{2}\sum_i\lambda_i(0)\\
\leq C\langle v,A^{cross}_Nv\rangle+\langle v,Op_W^{N^{-1}}(Q_F(0)(x,\xi))v\rangle +C\langle v,Op^{N^{-1}}_W(|N^{-\frac
  12},x,\xi|^3)v\rangle\\ +aC\left(\langle
  v,Op_W^{N^{-1}}(|x|^2+|\xi|^2)v\rangle
  -N^{-1}\frac{n-r}{2}\right)+CN^{-\frac 43}.
\end{multline*}

\end{prop}
\begin{proof}
  The proof follows the exact same lines as for Proposition
  \ref{prop:Weylreg}: the difficulty lies in handling the $(x,\xi)$ terms
  which take a similar form as above.

  The supplementary $N^{-\frac 43}$ terms are due to positivity
  estimates for the Weyl quantization: from $c\sigma(A_N^{cross})\leq
  g_0$ we can only deduce $cA_N^{cross}\leq
  Op_W^{N^{-1}}(g_0)+O(N^{-\frac 43})$.
\end{proof}

\subsection{Spectral gap}
As before, we show that the almost eigenfunction found previously
corresponds to the first eigenvalue.

\begin{prop}\label{prop:order-cross}
  Let $h\geq 0$ be such that the minimum of the Melin value $\mu$ is only reached at
  one point, which is a simple crossing point of $h$.
  
  Let $(\mu_i)$ be the real sequence constructed in
  Proposition \ref{prop:eigenf-cross}, and let $\lambda_{min}$ be the first eigenvalue of
  $T_N(h)$. 

Then
\[
\lambda_{min}\sim N^{-1}\sum_{i=0}^{\infty}N^{-\frac
  i6}\mu_i.
\]
Moreover, there exists $c>0$ such that,
for every $N$, one has
\[
\dist(\lambda_{min}),Sp(T_H(h))\setminus
\{\lambda_{min}\})\geq c N^{-\frac 43}.
\]
\end{prop}

\begin{preuve}
  Let us show that any function orthogonal to the one proposed in
  Proposition \ref{prop:eigenf-cross} has an energy which is larger by at least
  $cN^{-\frac 43}$.

Let $(v_N)_{N\geq 1}$ be a sequence of unit vectors in $L^2(\R^n)$. If $\langle
v_N,Op^{1}_W(\tilde{g}_N)v_N\rangle \leq N^{-1}\mu_0+CN^{-\frac
  43}$ for some $C$, then $v_N=e^{-\frac{|x|^2}{2}}w_N(q)+O(N^{-\frac 13})$, with
$\|w_N\|_{L^2}=1+O(N^{-\frac 13})$.

If $C-\mu_2$ is strictly smaller than the spectral gap of the
operator $P$ then
$\langle w_N,v_0\rangle\geq a$ for some $a>0$,
which concludes the proof.
\end{preuve}


\section{Comparative Weyl law}

\begin{defn}
  \label{defn:dim}
We will say a miniwell has \emph{dimension} $r$ when the dimension of
the zero set of $h$ around the miniwell is $r$. Similarly, we will
say a crossing point has dimensions $(r_1,r_2)$ when the dimensions of
the two manifolds $Z_1$ and $Z_2$ around the point are $r_1$ and
$r_2$, respectively.
\end{defn}

\begin{proof}[Proof of theorem D]$\,$\\
  \begin{enumerate}[label=\Alph*.]
    \item
  The first statement stems directly from Proposition
  \ref{prop:plr}. In fact, the Subsection 4.2 concludes the proof of
  Theorem A using only the fact that the considered eigenvalue lies in
  this spectral window for every $\epsilon>0$, hence the claim.
  
  \item Let $u$
    a sequence of eigenfunctions of $T_N(h)$ in the spectral window
    above, and $P_0$ the miniwell of interest. The
    sequence $u$ localises near $P_0$ so that one can apply a quantum
    map $\mathfrak{S}_N$. 

Let $v_N=\mathcal{B}_N^{-1}\mathfrak{S}_N^{-1}u_N$. The first lower bound on Proposition \ref{prop:Weylreg} yields
\[
\langle v_N,Op_W^{N^{-1}}(|x|^2+|\xi|^2)v_N\rangle-N^{-1}\frac{n-r}2\leq
C\Lambda_NN^{-1}.\]

From Proposition \ref{prop:plr}, for every $\delta>0$, if $\epsilon>0$ is small enough then
$u$ localises on $B(P_0,\delta)$. If $\delta$ is
small enough then
\[\delta C\left(\langle
    v_N,Op_W^{N^{-1}}(|x|^2+|\xi|^2)v_N\rangle-N^{-1}\frac{n-r}2\right)\leq
  \frac{N^{-1}\Lambda_N}{2}.\]

Let us prove an upper bound in the number of eigenvalues of
$T_N(h)$. The second lower bound in Proposition \ref{prop:Weylreg}
leads to
\begin{multline*}
  c\langle v_N,A_N^{reg}v_N\rangle + \langle
  v_N,Op_W^{N^{-1}}(Q_F(0)(x,\xi))v_N\rangle
  -\frac{N^{-1}}2\sum_i\lambda_i(0)\\\leq \langle
  u_N,hu_N\rangle-N^{-1}\mu(P_0)+\frac{N^{-1}\Lambda_{N}}2.
  \end{multline*}
For $\epsilon$ smaller than the spectral gap of $Q_F(0)(x,D)$, the left-hand side has less than $C\Lambda_N^r$ eigenvalues smaller
than $\frac{3N^{-1}\Lambda_N}{2}$, hence the claim.

The lower bound proceeds along the same lines. The upper bound in
Proposition \ref{prop:Weylreg} yields
\begin{multline*}C\langle v_N,A_N^{reg}v_N\rangle+ \langle
  v_N,Op_W^{N^{-1}}(Q_F(0)(x,\xi))v_N\rangle -
  \frac{N^{-1}}2\sum_i\lambda_i(0)\\\geq \langle u_N,hu_N\rangle
  -N^{-1}\mu(P_0)-\frac{N^{-1}\Lambda_N}2.
\end{multline*}
The left-hand side has always more than $c\Lambda_N^r$ eigenvalues
smaller than $\frac{N^{-1}\Lambda_N}{2}$, hence the claim.
\item The proof for crossing points is the same except for the actual
  count of eigenvalues of the reference operator, which stems from
  Proposition \ref{prop:Weyl}.
\end{enumerate}
\end{proof}


\section{Examples: frustrated spin systems}

In this Section, we discuss the class of examples introduced
in Subsection \ref{sec:appl-spin-syst}. We first describe the minimal
set in the general setting, and we prove that, for
a loop of six triangles, the classical minimal set is not a smooth
manifold; then we prove that the choice of the vectors
on triangle ``leaves'' does not affect $\mu$; to conclude we treat numerically a simple case supporting
the general conjecture that $\mu$ is minimal on planar configurations.

\subsection{Description of the zero set}

If a graph is made of triangles $(V_i)_{i\in J}$, and if we denote by
$\{u_i,v_i,w_i\}$ the three elements of $\S^2$ at the vertices of
$V_i$, we write
$$h(e)=\sum_{i\in J}u_i\cdot v_i + u_i\cdot w_i+v_i\cdot w_i.$$
Moreover, for all $u,v,w\in \S^2$ one has
$$u\cdot v + u\cdot w + v\cdot w=\frac 12 \|u+v+w\|^2-\frac 32.$$

A way to minimize the symbol is thus to try to choose the vectors such
that, for each triangle in the graph, the vectors at the vertices form
a great equilateral triangle on $\S^2$ (this is equivalent to the
requirement that their sum is the zero vector). As the example of the
Husimi tree shows, this minimal set can be degenerate: once the vector
at a vertex is chosen, there is an $\S^1$ degeneracy in the
choice of the vectors at its children. 

In the general case this solution is not always possible as can be
seen on the right of Figure \ref{fig:Huskag}. Moreover, even if this solution is possible, the
minimal set is not a submanifold, as we will see in an example.

A subset of interest of these minimal configurations consists in the
case where all vectors are coplanar; this corresponds to colouring the
graph with three colours. For some graphs made of triangles, there is no
3-colouring. Conversely, if the size of the graph grows the number of
3-colourings may grow exponentially fast.

A common conjecture in the physics literature is that, when
applicable, the Melin value $\mu$ is always minimal only
on planar configurations, except for a leaf degeneracy (see
Proposition \ref{prop:leaf}): in other terms, in the semiclassical limit,
the quantum state presumably selects only planar configurations. It is unclear
whether a study of the sub-subprincipal effects would discriminate
further between planar configurations, but numerical evidence suggests
that the quantum ground state is not distributed evenly on them at
large spin.

Other selection effects tend to select the planar configurations:
consider for instance the classical Gibbs measure, at a very small
temperature. This measure concentrates on the points of the minimal
set where the Hessian has a maximal number of zero eigenvalues
(thermal selection); in
this case it always corresponds to planar configurations, if any.

\subsection{Irregularity of the zero set}
One of the key examples of frustrated spin systems is the Kagome
lattice. We restrict our study to the case of one
loop of six triangles. 

\begin{prop}
  For a loop of six triangles (as in Figure \ref{fig:6tri}), the
  minimal set is not smooth.
\end{prop}
\begin{proof}
  The
choice of the two vectors drawn on the left in Figure \ref{fig:6tri} induces a global
$SO(3)$ rotation, and without loss of generality we will keep them
fixed. Moreover, the position of the six inner vectors determines the
position of the six outer vectors in a unique and smooth way, so we
will forget about the latter.

\begin{figure}
\centering
\includegraphics[scale=0.4]{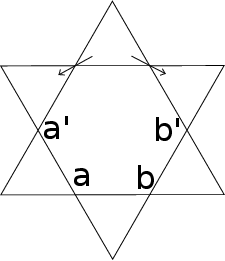}
\includegraphics[scale=0.4]{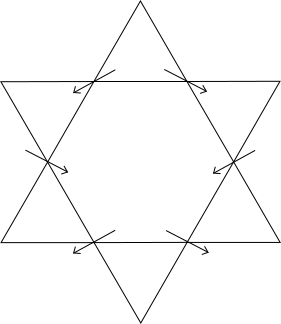}
\caption{On the left, a graph with 6 triangles and two prescribed
  vectors. On the right, a particular (planar) configuration.}
\label{fig:6tri}
\end{figure}

The space of configurations of the pair $(a,a')$ is a subset of a
two-dimensional torus; indeed the choice for $a'$ is made along a
circle with center having its center on the lower-left vector, and the choice for $a$ is
similarly made along a circle with center $a'$. The above applies to
the pair $(b,b')$. Hence, the set of global configurations is a subset of a
four-dimensional torus: the subset on which the angle between $a$ and
$b$ is exactly $\frac{2\pi}{3}$. This cannot be an open set of the
four-dimensional torus, as every coordinate and function involved is
real analytic. Hence, if this set is a smooth manifold, its dimension
does not exceed three.

On the other hand, consider the particular case of Figure \ref{fig:6tri} which represents a
particular configuration. From this configuration, one stays in the
minimal set by moving $a'$ along a circle with center $a$; one can also
move along $a$ only, or along $b$ only, or along $b'$ only. The set of
possible smooth moves from this configuration spans a set of dimension
at least four, hence the contradiction.

\end{proof}

\subsection{Degeneracy for triangle leaves}

The simplest example of a frustrated system is a triangle with three
vertices, connected with each other. In this setting the degeneracy of
the minimal set (which is exactly the set of configurations such that
the sum of the three vectors is zero) corresponds to a global $SO(3)$
symmetry of the problem; in this case the function $\mu$ is constant.

Consider the left part of Figure \ref{fig:1tri}. The three elements $e_1,e_2,e_3$ lie on the same
large circle. We choose the coordinate $q_i$ along this circle and the
coordinate $p_i$ orthogonal to it. In these coordinates, the
half-Hessian of the classical symbol can be written as:
$$2(p_1+p_2+p_3)^2+(q_1-q_2)^2+(q_1-q_3)^2+(q_2-q_3)^2.$$
Since this quadratic form does not depend on the positions of $e_1,e_2,e_3$,
the function $\mu$ is constant.
\begin{figure}
  \centering

  \includegraphics[width=0.3\textwidth]{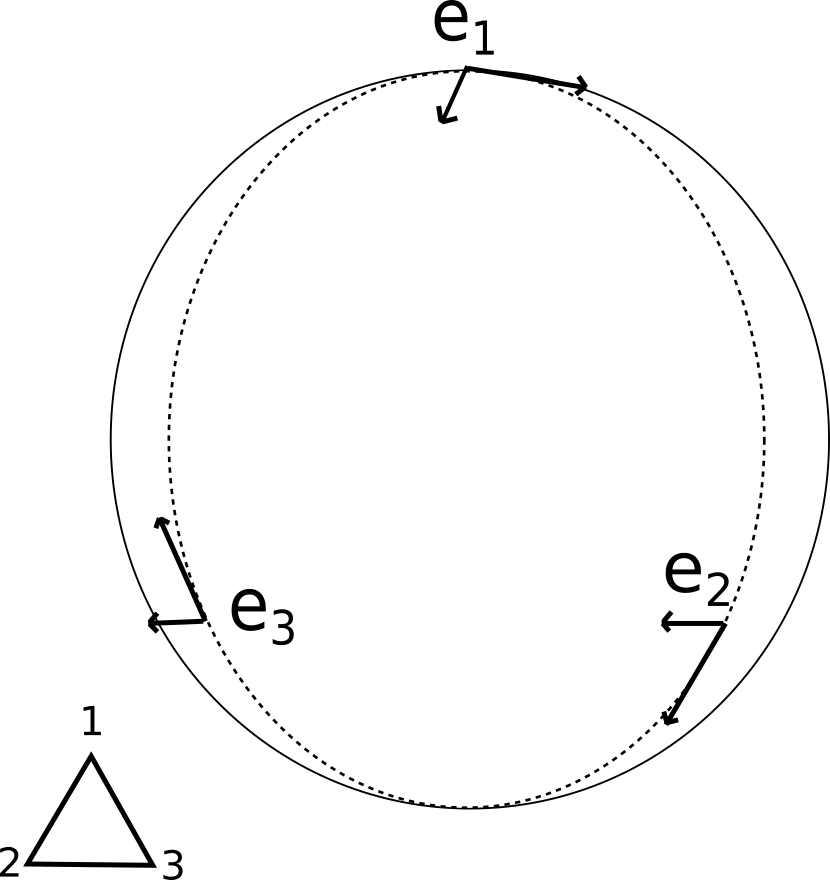}
\hspace{3em}
\includegraphics[width=0.3\textwidth]{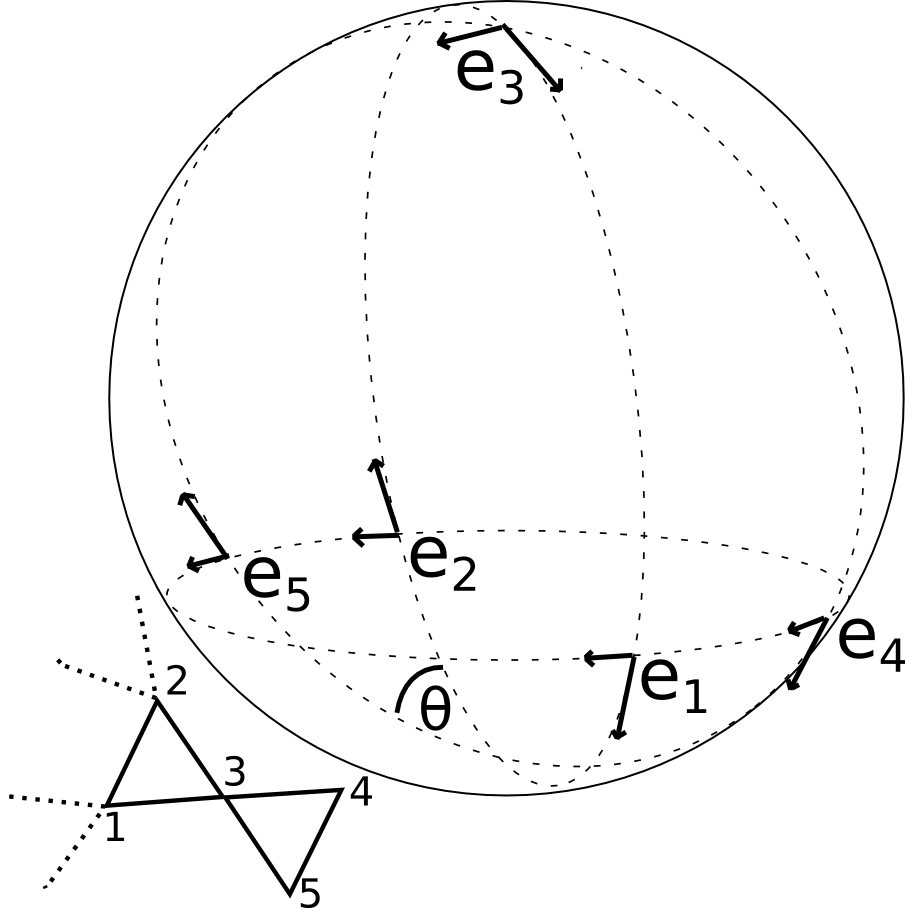}
\\
  \caption{General minimal configuration for one triangle (left) and an
     triangle leaf (right) of spins, with choice of tangent coordinates.}
  \label{fig:1tri}
\end{figure}

In the following Proposition we consider a slightly more general situation.
\begin{prop}\label{prop:leaf}
Consider a graph with a ``triangle leaf'' as in the inset on the
right of Figure 2. In order to find a classical minimum for such a graph, once
all vectors except for $e_4$ and $e_5$ are chosen, then $e_4$ and
$e_5$ are fixed except for a rotation of centre $e_3$.

The Melin value $\mu$ does not depend on this choice.
\end{prop}
\begin{proof}
Denoting $c=\cos(\theta)$ and $s=\sin(\theta)$, and using local
coordinates as in the right part of Figure \ref{fig:1tri}, the 2-jet of the
Hamiltonian reads, in local coordinates:
\begin{multline*}
Q(p_1,p_2,q_1,q_2,...)+2(p_4+p_5)^2+(q_4-q_5)^2+q_4^2+q_5^2+4q_3^2+4p_3^2\\
+4p_3(p_1+p_2)-2q_3(q_1+q_2)\\+4cp_3(p_4+p_5)-4sq_3(p_4+p_5)\\
-2cq_3(q_4+q_5)-2sp_3(q_4+q_5).
\end{multline*}
The trace of this quadratic form does not depend on
$\theta$. Hence, in order to prove that $\mu$ does not depend on
$\theta$ it is sufficient to find symplectic coordinates in which this
quadratic form does not depend on $\theta$.
A first symplectic change of variables leads to:
\begin{multline*}
  Q(p_1,p_2,q_1,q_2,...)+4p_4^2+q_4^2+3q_5^2+4q_3^2+4p_3^2\\
  +4p_3(p_1+p_2)-2q_3(q_1+q_2)\\
 +4\sqrt{2}cp_3p_4-4\sqrt{2}sq_3p_4-2\sqrt{2}cq_3q_4-2\sqrt{2}sp_3q_4.
\end{multline*}

Let us make the following change of variables:
\begin{align*}
  p_4&\mapsto cp_4-s\frac{q_4}{2}\\
  q_4&\mapsto cq_4+2sp_4
\end{align*}
This change of variables is symplectic, and preserves $4p_4^2+q_4^2$. The quadratic form becomes:
\begin{multline*}
  Q(p_1,p_2,q_1,q_2...)+4p_3(p_1+p_2)-2q_3(q_1+q_2)\\+4p_4^2+q_4^2+3q_5^2+4p_3^2+4q_3^2
+8p_3p_4-4q_4q_3.
\end{multline*}
Since this quadratic form does not depend on $\theta$, the function $\mu$ does not
depend on $\theta$.
\end{proof}

\subsection{A numerical example}
\label{sec:numerical-example}

The last example we treat is the case of a loop of 4 triangles. In this setting, the minimal set is not a submanifold but a
union of three submanifolds, with transverse intersection. The general
configuration is shown in Figure \ref{fig:4tri}. The
intersections correspond in fact to the case of crossing along a
submanifold (see Definition \ref{defn:cross-subm}), since the Hessian
matrix at each minimal point can be computed explicitly. From parity
properties we can deduce that the function $\mu$ reaches a
\emph{local} minimum on these crossings, however we cannot conclude
that the ground state selects this set.

The quadratic form is again explicit if the local coordinates are
chosen conveniently, but the computation of $\mu$ depends on an exact
diagonalisation which we believe to be less explicit than in the previous case.

 \begin{figure}
   \centering
   \includegraphics[scale=0.2]{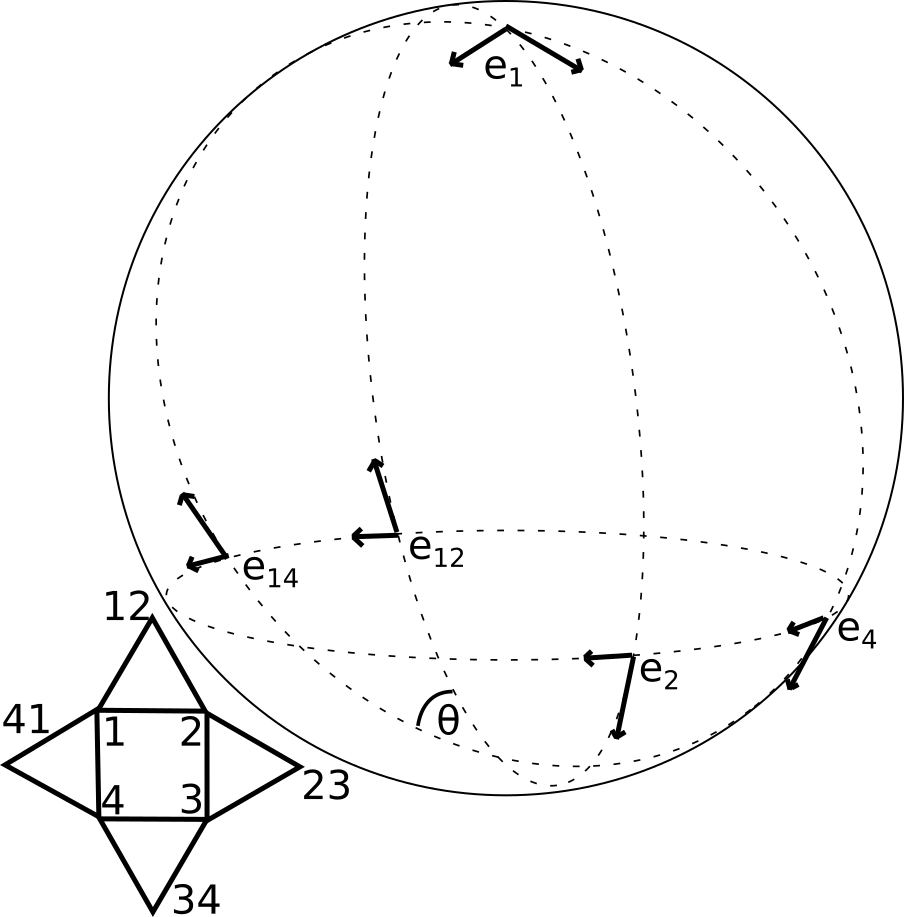}
   \includegraphics[scale=0.2]{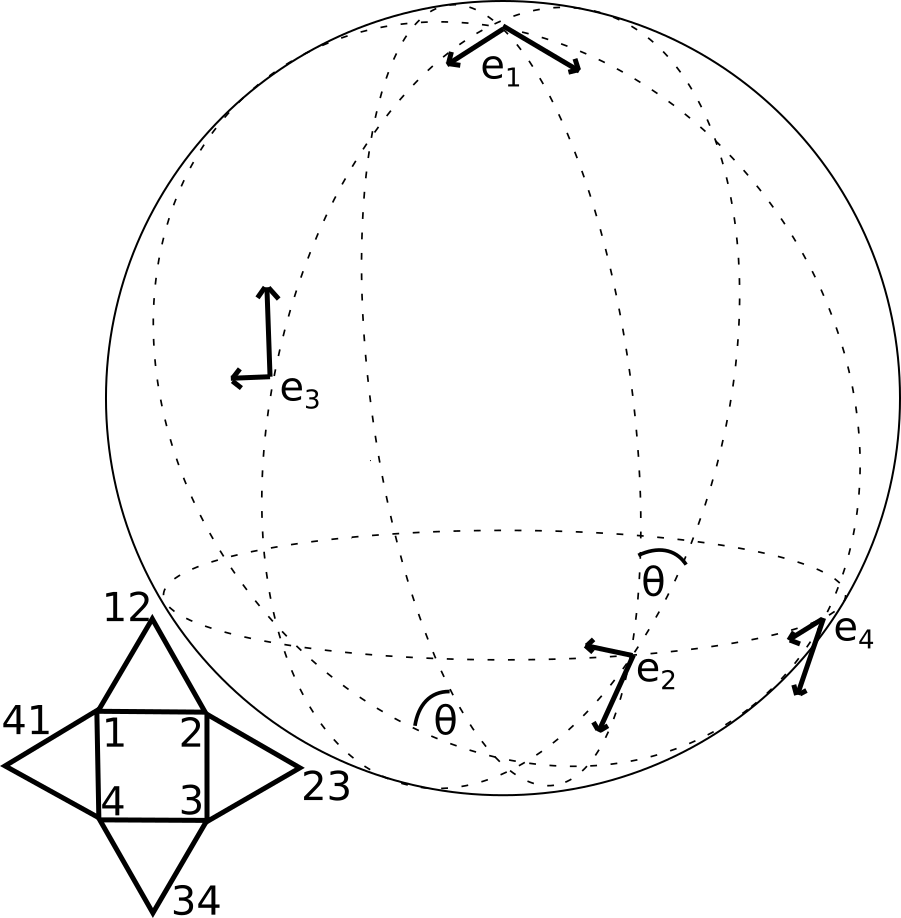}
   \caption{The two general configurations for a loop of 4
     triangles. For the first configuration, $e_3=e_1$ so
     $e_{12}=e_{23}$ and $e_{14}=e_{34}$. In the second
     configuration, for simplicity, we omitted to draw $e_{12},e_{23},e_{34}$
     and $e_{14}$.}
   \label{fig:4tri}
 \end{figure}






A numerical plot of $\mu$ as a function of $\theta$ is
presented in Figure \ref{fig:4tri-results}. Note that $\mu$ is not
smooth on the crossing point, in accordance with Definition \ref{defn:cross-subm}. We believe that a closed expression of $\mu$ is, in
this case, rather technical to obtain. Figure
\ref{fig:4tri-results} is a strong indication that $\mu$ is only
minimal on flat configurations.
\begin{figure}
  \centering
  \includegraphics[width=0.49\textwidth]{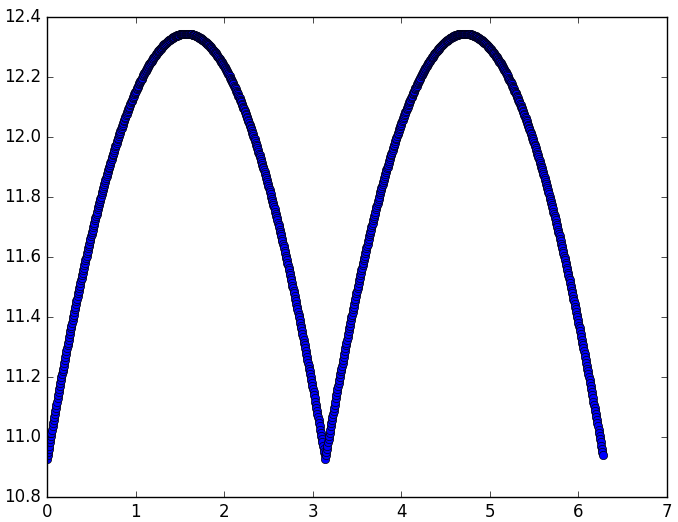}
  \includegraphics[width=0.49\textwidth]{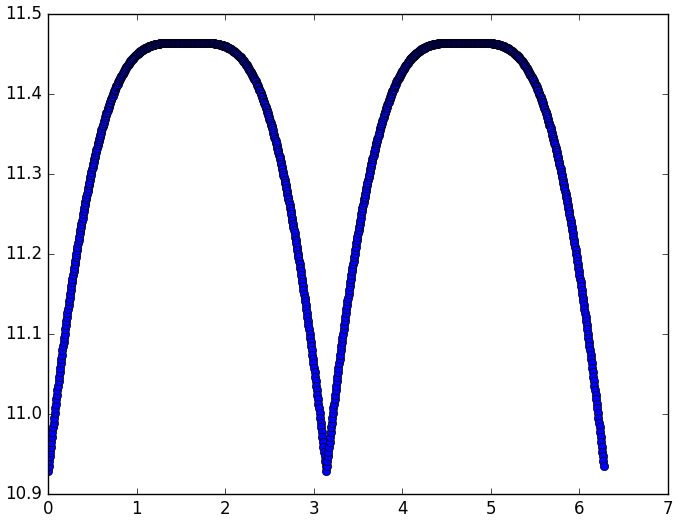}

  \caption{Numerical plot of the value of $\mu$ as a function of
    $\theta$ in the two situations corresponding to figure
    \ref{fig:4tri}. In both cases the minimum is reached on planar
    configurations, with a minimal value $\mu \simeq 10.928$.}
  \label{fig:4tri-results}
\end{figure}









\bibliographystyle{abbrv}
\bibliography{math}
\end{document}